\documentclass[letterpaper,11pt,reqno]{amsart}
\usepackage[active]{srcltx}
\usepackage{longtable}
\usepackage{bm}
\usepackage{amsmath,amssymb,amscd, booktabs,mathtools,commath}
\usepackage[all]{xy}
\usepackage[headings]{fullpage}

\let\oldtocsection=\tocsection

\let\oldtocsubsection=\tocsubsection

\let\oldtocsubsubsection=\tocsubsubsection

\renewcommand{\tocsection}[2]{\hspace{0em}\oldtocsection{#1}{#2}}
\renewcommand{\tocsubsection}[2]{\hspace{1em}\oldtocsubsection{#1}{#2}}
\renewcommand{\tocsubsubsection}[2]{\hspace{2em}\oldtocsubsubsection{#1}{#2}}
\newcommand{\forget}[1]{}

\allowdisplaybreaks

\newtheorem{lemma}{Lemma}[section]

\newtheorem*{localtwisttheorem}{Local Twisting Theorem}
\newtheorem*{paratwisttheorem}{Paramodular Twisting Theorem}
\newtheorem*{hecketheorem}{$L$-function Theorem}

\usepackage[pdftex, colorlinks, linkcolor={blue}, citecolor={blue} ,pdfpagelabels, 
hyperindex, plainpages, pageanchor ]{hyperref}

\newcommand{\A}{{\mathbb A}}
\newcommand{\Q}{{\mathbb Q}}
\newcommand{\Z}{{\mathbb Z}}
\newcommand{\R}{{\mathbb R}}
\newcommand{\C}{{\mathbb C}}

\newcommand{\p}{\mathfrak p}
\newcommand{\OF}{{\mathfrak o}}
\newcommand{\GL}{{\rm GL}}

\newcommand{\SL}{{\rm SL}}

\newcommand{\GSp}{{\rm GSp}}

\newcommand{\val}{{\rm val}}

\newcommand{\SSp}{{\rm Sp}}

\begin{document}
 
\title{Twisting of Siegel paramodular forms}
\author[Johnson-Leung and Roberts]{Jennifer Johnson-Leung\\
Brooks Roberts}
\begin{abstract}
Let $S_k(\Gamma^{\mathrm{para}}(N))$ be the space of Siegel paramodular forms of level $N$ and weight $k$.  Fix an odd prime $p\nmid N$ and let $\chi$ be a nontrivial quadratic Dirichlet character mod $p$.  Based on \cite{JR1}, we define a linear twisting map $\mathcal{T}_\chi:S_k(\Gamma^{\mathrm{para}}(N))\rightarrow S_k(\Gamma^{\mathrm{para}}(Np^4))$.  We calculate an explicit expression for this twist, give the commutation relations of this map with the Hecke operators and Atkin-Lehner involution for primes $\ell\neq p$, and prove that the $L$-function of the twist has the expected form.
\end{abstract}
\maketitle

\section{Introduction}
Let $k$ and $N$ be positive integers and let $p$ be an odd prime with $p\nmid N$.  Let $S_k(\Gamma_0(N))$ denote the space of elliptic modular cusp forms of weight $k$ with respect to $\Gamma_0(N)\subset \SL(2,\Z)$ and let $\chi$ be a nontrivial quadratic Dirichlet character mod $p$.  There is a natural twisting map $\tau_\chi:S_k(\Gamma_0(N))\rightarrow S_k(\Gamma_0(Np^2))$  such that if $f\in S_k(\Gamma_0(N))$ then 
\begin{equation}\label{gl2twisteq}
\tau_\chi(f)=\sum_{u\in(\Z/p\Z)^\times}\chi(u)f|_k\begin{bmatrix}1&u/p\\&1\end{bmatrix}.
\end{equation}

See, for example, \cite{Sh} Proposition 3.64.  A calculation verifies that for a prime $\ell\neq p$, 
\begin{equation}\label{gl2heckeeq}
T(\ell)\tau_\chi=\chi(\ell)\tau_\chi T(\ell)\qquad\text{and}\qquad (\tau_\chi f)|_kW_\ell=\chi(\ell)^{\val_\ell(N)}\tau_\chi(f|_k W_\ell),
\end{equation}
where $f\in S_k(\Gamma_0(N))$, $T(\ell)$ is the Hecke operator, and $W_\ell$ is the Atkin-Lehner involution at $\ell$ as defined in Section 2.4 of \cite{Cr}, for example.  Moreover, if $f\in S^{\mathrm{new}}_k(\Gamma_0(N))$ is a Hecke eigenform for all primes, then $\tau_\chi(f) \in S^{\mathrm{new}}_k(\Gamma_0(Np^2))$ is also a Hecke eigenform for all primes and
\begin{equation}\label{gl2lfunction}
R(s,\tau_\chi(f))=W(\chi)R(s,f,\chi),
\end{equation}
where  $W(\chi)=\sum_{u\in(\Z/p\Z)^\times}\chi(u)e^{2\pi i u/p}$ is the Gauss sum of $\chi$, $R(s,\tau_\chi(f))=R(s, \tau_\chi(f),1)$ and $R(s,f,\chi)$ is the completed $L$-function as defined in \cite{Sh}.  By identifying $S_k(\Gamma_0(N))$ with automorphic forms on the ad\`eles of $\GL(2)$ over $\Q$, it is evident that this twisting only acts on the automorphic form at the prime $p$.  Hence the global twist is induced by a local twisting map for representations of $\GL(2,\Q_p)$.  

In this paper, we define and investigate the twisting map $\mathcal{T}_\chi$ on the space of Siegel paramodular forms induced by the local twisting map for paramodular representations of $\GSp(4,\Q_p)$ with trivial central character defined in \cite{JR1}.  The main results of this work are given in Section \ref{maintheoremsec}. Our \hyperlink{mainsiegeltheorem}{Paramodular Twisting Theorem} proves a formula for $\mathcal{T}_\chi$ similar to, but more involved than, the formula \eqref{gl2twisteq}.  Our \hyperlink{heckeoperatortheorem}{$L$-function Theorem} proves commutation relations for the paramodular Hecke operators and Atkin-Lehner involution analogous to \eqref{gl2heckeeq} and an identity of completed $L$-functions analogous to \eqref{gl2lfunction}.  In our view, this identity demonstrates that $\mathcal{T}_\chi$ is in fact canonical. Moreover, it follows from Theorem 1.2 of \cite {JR1} that the map $\mathcal{T}_\chi$ is not zero in general.  Hence, our theorem may provide a source of examples to study conjectures such as the paramodular conjecture \cite{BK} and the paramodular B\"ocherer's conjecture \cite{RT}.  

In another paper \cite{JR2}, we apply the formula \eqref{twistslasheq} in the  \hyperlink{mainsiegeltheorem}{Paramodular Twisting Theorem} to compute Fourier coefficients of the twisted paramodular form $\mathcal{T}_\chi(F)$, resulting in a formula analogous to but more complicated than the formula for an elliptic modular form:
$$
\tau_\chi(f)(z)=\sum_{n=1}^\infty W(\chi)\chi(n)a_ne^{2\pi i n z},\quad \text{where}\quad f(z)=\sum_{n=1}^\infty a_ne^{2\pi i n z}.
$$   
In \cite{JR2} we also explicitly show that the Fourier coefficients of the twist of a Maass form are identically zero, as expected from local considerations.  This vanishing is a result of nontrivial cancellations that vary depending on the index of the Fourier coefficient.  In our view, this provides strong evidence that the formulas of \eqref{twistslasheq} are correct and do not admit further significant reductions.

Other types of twists by Dirichlet characters of Siegel modular forms and their $L$-functions have been studied.  See for example, \cite{KR} and the references cited therein.  However, the goal of the current paper as explained above appears to be different from previous work.  Again, our goal is to calculate maps on Siegel paramodular forms corresponding to the canonical twists of automorphic representations by quadratic Dirichlet characters.  

\section{Notation}\label{notationsec}
Let $M$ be a positive integer and let $\chi:(\mathbb Z/ M \mathbb Z)^\times \to \mathbb C^\times$
be a Dirichlet character. We let $\A$ denote the adeles of $\Q$ and define an associated Hecke character $\mathbb Q^\times \backslash \A^\times \to \mathbb C^\times$, 
denoted by $\bm \chi$, as follows. Recall that 
$
\A^\times= \mathbb Q^\times \mathbb R_{>0}^\times \prod_{\ell < \infty} \mathbb Z_\ell^\times,
$
with the groups embedded in $\A^\times$ in the usual ways. In fact, the map defined  by $(q,r,n)\mapsto qrn$
defines an isomorphism of topological groups
$$
\mathbb Q^\times \times \mathbb R_{>0}^\times \times \prod_{\ell < \infty} \mathbb Z_\ell^{\times} \stackrel{\sim}{\longrightarrow} \A^\times.
$$
Let $M=\ell_1^{\val_{\ell_1}(M)} \dots \ell_t^{\val_{\ell_t}(M)}$ be the prime factorization of $M$. We consider the composition
\begin{multline*}
\Z_{\ell_1}^\times\  \times \dots \times \  \Z_{\ell_t}^\times 
\longrightarrow
\Z_{\ell_1}^\times/(1+\ell_1^{\val_{\ell_1}(M)}\Z_{\ell_1}) \  \times \cdots \times \ \Z_{\ell_t}^\times/(1+\ell_t^{\val_{\ell_t}(M)}\Z_{\ell_t})\\
\stackrel{\sim}{\longrightarrow}
(\Z/ \ell_1^{\val_{\ell_1}(M)} \Z)^\times\  \times \dots \times\  (\Z/ \ell_t^{\val_{\ell_t}(M)} \Z)^\times
\stackrel{\sim}{\longrightarrow}
(\Z/M\Z)^\times
\stackrel{\chi}{\longrightarrow} \C^\times. 
\end{multline*}
We denote the restriction of this composition to $\Z_{\ell_i}^\times$ by $\bm \chi_{\ell_i}$. If $\ell \nmid M$, then we define $\bm \chi_\ell:\Z_\ell^\times \to \C^\times$ to be the trivial character. 
For each finite prime $\ell$, $\bm \chi_\ell$ is a continuous character of $\Z_\ell^\times$, and $\bm \chi_\ell(1+\ell^{\val_\ell(M)}\Z_\ell)=1$. 
We define the corresponding Hecke character $\bm \chi: \mathbb Q^\times \backslash \A^\times \to \mathbb C^\times$ as the composition
\begin{equation}
\label{chitochieq}
\mathbb A^\times \stackrel{\sim}{\longrightarrow} \mathbb Q^\times \times \mathbb R_{>0}^\times \times \prod_{\ell < \infty} \mathbb Z_\ell^\times
\stackrel{\mathrm{proj}}{\longrightarrow} \prod_{\ell < \infty} \mathbb Z_\ell^\times \stackrel{\prod \bm \chi_\ell}{\longrightarrow}\mathbb C^\times.
\end{equation}
We see that if $a \in \Z$ with $(a,M)=1$, then 
\begin{equation*}\label{chichieq}
 \chi(a)=\bm \chi_{\ell_1}(a) \cdots \bm \chi_{\ell_t}(a).
\end{equation*}
Let 
$$
J=\begin{bmatrix}&\bf{1}_2\\-\bf{1}_2&\end{bmatrix}.
$$
We define the algebraic $\mathbb{Q}$-group $\GSp(4)$ as the set of all $g\in\GL(4)$ such that ${}^tgJg=\lambda(g)J$ for some $\lambda(g)\in\GL(1)$ called the multiplier of $g$. Let $\GSp(4,\R)^+$ be the subgroup of $g\in\GSp(4,\R)$ such that $\lambda(g)>0$. The kernel of $\lambda:\GSp(4)\rightarrow\GL(1)$ is the symplectic group $\SSp(4)$.
Let $N$ and
$k$ be  positive integers. We define the paramodular group of level $N$ to be 
$$
\Gamma^{\mathrm{para}}(N)=\SSp(4,\Q)\cap\begin{bmatrix}\Z&\Z&N^{-1}\Z&\Z\\N\Z&\Z&\Z&\Z\\N\Z&N\Z&\Z&N\Z\\N\Z&\Z&\Z&\Z\end{bmatrix}.
$$
We also define local paramodular groups.  For $\ell$ a prime of $\Q$ and $r$ an non-negative integer, let $K^{\mathrm{para}}(\ell^r)$ be the paramodular subgroup of $\GSp(4,\Q_\ell)$ of level $\ell^r$, i.e., the subgroup of elements $g \in \GSp(4,\Q_\ell)$ such that $\lambda (g) \in \Z_\ell^\times$ and 
$$
g \in \begin{bmatrix}
       \Z_\ell&\Z_\ell&\ell^{-r}\Z_\ell&\Z_\ell\\
\ell^r\Z_\ell&\Z_\ell&\Z_\ell&\Z_\ell\\
\ell^r\Z_\ell&\ell^r\Z_\ell&\Z_\ell&\ell^r\Z_\ell\\
\ell^r\Z_\ell&\Z_\ell&\Z_\ell&\Z_\ell
      \end{bmatrix}.
$$Note that
$$
\Gamma^{\mathrm{para}}(N) = \GSp(4,\Q) \cap \GSp(4,\R)^+ \prod_{\ell<\infty} K^{\mathrm{para}}(\ell^{\val_\ell(N)}),
$$
with intersection in $\GSp(4,\A)$.  

For $n$ a positive integer, let $\mathfrak{H}_n$ denote the Siegel upper half plane of degree $n$ with $I=i{\bf1}_{2n}$.  The group $\GSp(4,\R)^+$ acts on $\mathfrak{H}_2$ via 
$$
h\langle Z\rangle=(AZ+B)(CZ+D)^{-1}, \qquad h=\begin{bmatrix}A&B\\C&D\end{bmatrix} ,\qquad Z\in\mathfrak{H}_2.
$$
Denote the factor of automorphy by $j(h,Z)=\det(CZ+D)$. If $F : \mathfrak{H}_2 \to \C$ is a function and $h\in \GSp(4,\R)^+$, then we define $F|_k h: \mathfrak{H}_2 \to \C$ by
$$
(F|_k h)(Z) = \lambda(h)^k j(h,Z)^{-k} F(h\langle Z\rangle ),\qquad Z \in \mathfrak{H}_2.
$$
Let $\Gamma\subset\GSp(4,\Q)$ be a group commensurable with $\SSp(4,\Z)$.  We define $S_k(\Gamma)$ to be
the complex vector space of  functions $F: \mathfrak{H}_2 \to \C$
such that 
\begin{enumerate}
\item $F$ is holomorphic;
\item $F|_k \gamma = F$ for all $\gamma \in \Gamma$; 
\item $\lim_{t \to \infty} (F|_k \gamma )( \begin{bmatrix} it & \\ & z \end{bmatrix} ) =0$
for all $\gamma \in \SSp(4,\Z)$ and $z \in \mathfrak{H}_1$.
\end{enumerate}
For further background see, for example, \cite{PY}.

 Let $p$ be an odd prime with $p\nmid N$ and $\chi$ the non-trivial quadratic Dirichlet character modulo $p$.  Let $F\in S_k^{\mathrm{new}}(\Gamma^{\mathrm{para}}(N))$, and assume that for a prime number $\ell$,
$$
T(1,1,\ell,\ell)F=\lambda_{F,\ell}F,\qquad T(1,\ell,\ell,\ell^2)F=\mu_{F,\ell}F,\qquad F|_kU_\ell=\varepsilon_{F,\ell}F,
$$  
 where the paramodular Hecke operators and Atkin-Lehner involution are defined below in Section \ref{heckeoperators}.  Motivated by the results in \cite{RS}, we define $L_\ell(s,F,\chi)$ as follows:\\

\noindent i) If $\val_\ell(Np)=0$, then
$$
L_\ell(s,F,\chi)^{-1}=1-\chi(\ell)\lambda_{F,\ell}\ell^{-s}+(\ell\mu_{F,\ell}+\ell^{2k-3}+\ell^{2k-5})\ell^{-2s}-\ell^{2k-3}\chi(\ell)\lambda_{F,\ell}\ell^{-3s}+\ell^{4k-6}\ell^{-4s}.
$$
\noindent ii) If $\val_\ell(N)=1$, then
$$
L_\ell(s,F,\chi)^{-1}=1-\chi(\ell)(\lambda_{F,\ell}+\ell^{k-3}\varepsilon_{F,\ell})\ell^{-s}+(\ell\mu_{F,\ell}+\ell^{2k-3})\ell^{-2s}+\chi(\ell)\varepsilon_{F,\ell}\ell^{3k-5}\ell^{-3s}.
$$
\noindent iii) If $\val_\ell(N)\geq2$, then
$$
L_\ell(s,F,\chi)^{-1}=1-\chi(\ell)\lambda_{F,\ell}\ell^{-s}+(\ell\mu_{F,\ell}+\ell^{2k-3})\ell^{-2s}.
$$
\noindent iv) If $\ell=p$, then $L_\ell(s,F,\chi)=1$.\\

\noindent If $F$ is an eigenform for every prime number $\ell$, we define the completed $\chi$-twisted $L$-function of $F$ to be
$$
\Lambda(s,F,\chi)=2\pi^{-2s}\Gamma(s)\Gamma(s-k+2)\prod_{\ell<\infty}L_\ell(s,F,\chi).
$$

Let $Q$ be a $2\times 2$ symmetric matrix, and let $P$ be an invertible $2\times 2$ matrix.  Then the matrices
$$U(Q)=\begin{bmatrix}1&Q\\&1\end{bmatrix}\qquad\text{and}\qquad A(P)=\begin{bmatrix}P&\\&{}^tP^{-1}\end{bmatrix}.$$  are in $\GSp(4)$.

\section{Main results}\label{maintheoremsec}
In this section, we present the main results of this paper.  The proofs of these theorems are presented in subsequent sections. The local twisting map from \cite{JR1} induces a twisting map for Siegel paramodular forms, as explained in Section \ref{automorphicformssection}. This theorem expresses the twisting map as a slash operator in a fashion analogous to the formula \eqref{gl2twisteq} for elliptic modular cusp forms.  Significantly, the matrices in the formula lie in the Borel subgroup of $\GSp(4)$, allowing, for example, the computation of Fourier coefficients of the twist, as in \cite{JR2}.
\begin{paratwisttheorem}
\hypertarget{mainsiegeltheorem}{Let} $N$ and $k$ be positive integers, $p$ an odd prime with $p\nmid N$, and $\chi$ the nontrivial quadratic Dirichlet character mod $p$.  Then, the local twisting map from Theorem 1.2 of \cite{JR1} induces a linear map
$$
\mathcal{T}_\chi:S_k(\Gamma^{\mathrm{para}}(N))\rightarrow S_k(\Gamma^{\mathrm{para}}(Np^4)).
$$
If $F\in S_k(\Gamma^{\mathrm{para}}(N))$, then this map is given by the formula
\begin{equation}\label{twistslasheq}
\mathcal{T}_\chi(F)=\sum_{i=1}^{14}F|_k\mathcal{T}^i_\chi,
\end{equation}
where the $\mathcal{T}^i_\chi$ are defined from the $T^i_\chi$ in the \hyperlink{localtwisttheorem}{Local Twisting Theorem} below by replacing $\OF$ by $\Z$, $\p^k$ by $p^k\Z$, $\varpi$ and $q$ by $p$, and finally $A(P)U(Q)$ by $U(-Q)A(P^{-1})$ for $2\times2$ matrices $P$ and $Q$.  We also extend the slash operator $|_k$  to formal $\C$-linear combinations of elements of $\GSp(4,\R)^+$.
\end{paratwisttheorem}

\noindent The proof of the \hyperlink{paratwisttheorem}{Paramodular Twisting Theorem} is given in Section \ref{maintheoremproofsubsec}.

The following theorem shows that the twisting map from the previous theorem is canonical.  Indeed, we prove that if $F\in S_k^\mathrm{new}(\Gamma^{\mathrm{para}}(N))$ is a Hecke eigenform and $\mathcal{T}_\chi(F)\neq0$, then $\mathcal{T}_\chi(F)$ is a newform and a Hecke eigenform with the expected $L$-function.  
\begin{hecketheorem}
\hypertarget{heckeoperatortheorem}{Let} $N$ and $k$ be positive integers, $p$ an odd prime with $p\nmid N$, and $\chi$ the nontrivial quadratic Dirichlet character mod $p$.  
\begin{enumerate}
\item For every prime $\ell\neq p$ and for $F\in S_k(\Gamma^{\mathrm{para}}(N))$  we have the following commutation relations for the Hecke operators \eqref{heckeoperatordef} and Atkin-Lehner operator \eqref{atkinlehnerdef}:
\begin{align*}
T(1,1,\ell,\ell)\mathcal{T}_\chi=&\chi(\ell)\mathcal{T}_\chi T(1,1,\ell,\ell),\\
T(1,\ell,\ell,\ell^2)\mathcal{T}_\chi=&\mathcal{T}_\chi T(1,\ell,\ell,\ell^2),\\
\mathcal{T}_\chi (F)|_kU_\ell=&\chi(\ell)^{\val_\ell(N)}\mathcal{T}_\chi(F|_kU_\ell),
\end{align*}
\item Let $F\in S^{\mathrm{new}}_k(\Gamma^{\mathrm{para}}(N))$ be an eigenform for $T(1,1,\ell,\ell)$, $T(1,\ell,\ell,\ell^2)$, and $U_\ell$ for every prime number $\ell$, and assume that $\mathcal{T}_\chi(F)\neq0$.  Then $\mathcal{T}_\chi(F)$ is in $S_k^\mathrm{new}(\Gamma^\mathrm{para}(Np^4))$ and is an eigenform for $T(1,1,\ell,\ell)$, $T(1,\ell,\ell,\ell^2)$, and $U_\ell$ for every prime number $\ell$, with Hecke eigenvalues
$$
\lambda_{\mathcal{T}_\chi(F),\ell}=\begin{cases}\chi(\ell)\lambda_{F,\ell}& \text{if }\ell\neq p\\
									0&\text{if }\ell=p,\end{cases}
\qquad
\mu_{\mathcal{T}_\chi(F),\ell}=\begin{cases}\mu_{F,\ell}&\text{if }\ell\neq p\\
-p^{2k-4}&\text{if }\ell=p,\end{cases}
$$ 
and
$$
\varepsilon_{\mathcal{T}_\chi(F),\ell}=\begin{cases}\chi(\ell)^{\val_\ell(N)}\varepsilon_{F,\ell}&\text{if }\ell\neq p\\1&\text{if }\ell=p.\end{cases}
$$
Moreover,
$$
\Lambda(s,\mathcal{T}_\chi(F))=\Lambda(s,F,\chi),
$$
where $\Lambda(s, \mathcal{T}_\chi(F))$ and $\Lambda(s,F,\chi)$ are the completed $L$-functions as defined in \cite{JR} and Section \ref{notationsec}, respectively.
\end{enumerate}
\end{hecketheorem}
\noindent The proof of the \hyperlink{hecketheorem}{$L$-function Theorem} is given in Section \ref{heckeoperators}.  Conjecturally, this completed $L$-function satisfies the functional equation
$$
\Lambda(2k-2-s,F,\chi)=(-1)^k(Np^4)^{s-k+1}\chi(N)\left(\prod_{\ell\mid N}\varepsilon_{F,\ell}\right)\Lambda(s,F,\chi).
$$
We note that this agrees with the equation in the main theorem of \cite {KR} in the case that the assumptions of that work and this work are both satisfied.

The following technical theorem expresses the local twisting map from \cite{JR1} in terms of matrices from the Borel subgroup. In the theorem, we let $F$ be a nonarchimedean local field of characteristic zero and odd residual characteristic with ring of integers $\OF$ and generator $\varpi$ of the maximal ideal $\p$ of $\OF$, and let $q$ be the number of elements in $\OF/\p$.  The existence of an expression involving only upper triangular matrices follows from the Iwasawa decomposition for $\GSp(4,F)$. However, as is evident from the proof, finding such an explicit expression requires some intricate calculations.
\begin{localtwisttheorem}
\hypertarget{localtwisttheorem}{Let} $(\pi,V)$ be a smooth representation of $\GSp(4,F)$ for which the center acts trivially, and let $\chi$ be a quadratic character of $F^\times$ of conductor $\p$.  Let $v\in V$ be such that $\pi(k)v=v$ for $k\in\GSp(4,\OF)$.  Then, the twisting map  $T_\chi:V(0) \to V(4,\chi)$ defined in Theorem 1.2 of \cite{JR1} is given by the formula
$$
T_\chi(v)=\sum_{i=1}^{14}\pi(T^i_\chi)v
$$
for $v \in V(0)$, where
\begin{align*}
T_\chi^1&=q^{-11}\!\!\!\!\!\!
\sum_{\substack{a,b,x\in(\OF/\p^3)^\times\\ z\in\OF/\p^4}}\chi(ab)   A(
\begin{bmatrix} 1 & -(a+xb)\varpi^{-1} \\&1 \end{bmatrix})U(\begin{bmatrix}  z \varpi^{-4} & b\varpi^{-2} \\b\varpi^{-2}&x^{-1}\varpi^{-1} \end{bmatrix}),\\
T_\chi^2&=q^{-11}\!\!\!\sum_{\substack{a,b\in(\OF/\p^3)^\times\\x,y\in(\OF/\p^3)^\times\\x,y\not\equiv1(\p)}}\!\!\!\chi(abxy)
A(\begin{bmatrix}
\varpi^{-1}&b \varpi^{-2} \\
&1
\end{bmatrix})U(\begin{bmatrix} ab(1-(1-y)^{-1}x)\varpi^{-3}&-a\varpi^{-2}  \\ -a\varpi^{-2}& ab^{-1}(1-x)^{-1}\varpi^{-1}\end{bmatrix}),\\
T_\chi^3&=q^{-6}\sum_{\substack{a\in(\OF/\p^2)^\times\\ b\in(\OF/\p^3)^\times\\ z\in(\OF/\p)^\times\\z\not\equiv1(\p)}}\chi(b(1-z))
A(\begin{bmatrix}\varpi^{-1}&\\&1\end{bmatrix})U(\begin{bmatrix}
b\varpi^{-3}&a\varpi^{-2}\\
a\varpi^{-2}&a^2b^{-1}z\varpi^{-1}
\end{bmatrix})v,\\
T_\chi^4&=q^{-10}\sum_{\substack{a\in\OF/\p^4\\ b\in(\OF/\p^3)^\times\\x\in(\OF/\p^4)^\times}} \chi(b)
A(\begin{bmatrix} \varpi^{-1}& x\varpi^{-3} \\ &1 \end{bmatrix})U(\begin{bmatrix} (b\varpi-ax)\varpi^{-4}& a\varpi^{-2} \\ a\varpi^{-2} &\end{bmatrix} )
,\\
T_\chi^5&=q^{-9}\sum_{\substack{a,b\in(\OF/\p^3)^\times\\x\in\OF/\p^3}}\chi(b)
A(\begin{bmatrix} \varpi^{-1}& x\varpi^{-2} \\ &1\end{bmatrix})U(\begin{bmatrix} (b-ax)\varpi^{-3}& a\varpi^{-2} \\ a\varpi^{-2} & \end{bmatrix}) 
,\\
T_\chi^6&=q^{-6}\sum_{\substack{a,b\in(\OF/\p^2)^\times\\ x\in(\OF/\p)^\times}} \chi(bx)
A(\begin{bmatrix}\varpi^{-2}&\\&1\end{bmatrix})U(
\begin{bmatrix}
b(1-x\varpi)\varpi^{-2}&a\varpi^{-2}\\
a\varpi^{-2}&a^2b^{-1}\varpi^{-2}
\end{bmatrix} )
,\\
T_\chi^7&=q^{-7}\sum_{\substack{a\in(\OF/\p^3)^\times\\ b\in(\OF/\p)^\times\\z\in\OF/\p^4}}\chi(ab)  
 A(\begin{bmatrix} 1 & -a\varpi^{-2} \\ &\varpi^{-1} \end{bmatrix})U( \begin{bmatrix} z\varpi^{-4} & b \varpi^{-1}\\ b\varpi^{-1} & \end{bmatrix}  ),\\
T_\chi^8&=q^{-9}\sum_{\substack{a,b,z\in(\OF/\p^3)^\times\\z\not\equiv1(\p)}}\chi(abz(1-z))
A(\begin{bmatrix}
\varpi^{-1}&b\varpi^{-3}\\
&\varpi^{-1}
\end{bmatrix})U(
\begin{bmatrix}
-ab(1-z)\varpi^{-3}  & a\varpi^{-1}\\
a\varpi^{-1}&
\end{bmatrix}
,\\
T_\chi^9&=q^{-6}\sum_{\substack{a\in(\OF/\p^2)^\times\\b,x\in(\OF/\p)^\times}} \chi(b)
A(
\begin{bmatrix}
\varpi^{-2}&a\varpi^{-3}\\
&\varpi^{-1}
\end{bmatrix})
U(\begin{bmatrix}
-b\varpi^{-1}& \\
 &x\varpi^{-1} 
\end{bmatrix},\\
T_\chi^{10}&=q^{-6}\sum_{\substack{a\in(\OF/\p^2)^\times\\ b\in(\OF/\p)^\times}} \chi(b)
A(\begin{bmatrix}
\varpi^{-2}&a\varpi^{-4}\\
&\varpi^{-2}
\end{bmatrix})U(\begin{bmatrix}b\varpi^{-1}&\\&\end{bmatrix}) ,\\
T_\chi^{11}&=q^{-10}
\sum_{\substack{a\in(\OF/\p^2)^\times\\b\in(\OF/\p^4)^\times\\x\in\OF/\p^3\\z\in\OF/\p^4}}\chi(ab)  A(\begin{bmatrix} 1 & b\varpi^{-1} \\ &\varpi\end{bmatrix})U(
\begin{bmatrix} z\varpi^{-4} & (xb+a\varpi)\varpi^{-3} \\(xb+a\varpi)\varpi^{-3}&-x\varpi^{-2} \end{bmatrix}  ,\\
T_\chi^{12}&=q^{-12}\!\!\!\!\!\!\!\!\!\sum_{\substack{y\in\OF/\p^4\\a\in(\OF/\p^4)^\times\\b,z\in(\OF/\p^3)^\times\\z\not\equiv1(\p)}} \!\!\!\!\!\!\!\chi(abz(1-z))
A(\begin{bmatrix}
\varpi^{-1}&a\varpi^{-2}\\
&\varpi\end{bmatrix})U(
\begin{bmatrix}
a(b(1-z)\varpi-y)\varpi^{-4}\!\!&y\varpi^{-3}\\
y\varpi^{-3}\!\! &-a^{-1}(y+b\varpi)\varpi^{-2}
\end{bmatrix}),\\
T_\chi^{13}&=q^{-6}\sum_{\substack{a\in(\OF/\p^2)^\times\\b\in(\OF/\p^3)^\times\\x\in(\OF/\p)^\times}} \chi(bx)A(\begin{bmatrix}\varpi^{-2}&\\&\varpi\end{bmatrix})U(\begin{bmatrix}b(1-x)\varpi^{-1}&a\varpi^{-2}\\
a\varpi^{-2}&a^2b^{-1}\varpi^{-3}
\end{bmatrix} )
,\\
T_\chi^{14}&=q^{-6}\sum_{\substack{a\in(\OF/\p^2)^\times\\b\in(\OF/\p)^\times\\x\in\OF/\p^4}} \chi(b)
A(\begin{bmatrix}\varpi^{-2}&\\&\varpi^2\end{bmatrix})U(
\begin{bmatrix}
-b\varpi^{-1}&a\varpi^{-2}\\
a\varpi^{-2}&x \varpi^{-4}\end{bmatrix} ).
\end{align*}
Here, $A(P)$ and $U(Q)$ are defined as in Section \ref{notationsec}.
\end{localtwisttheorem}
\noindent The proof of the \hyperlink{localtwisttheorem}{Local Twisting Theorem} is given in Section \ref{localsec}.
\section{Automorphic forms}\label{automorphicformssection}
In order to present the proof of the Paramodular Twisting Theorem we must explain the connection between Siegel modular forms and automorphic forms on $\GSp(4,\A)$.   Let $k$ be a positive integer, and let $\mathcal{F} = \{ K_\ell \}_\ell$, where $\ell$ runs over the finite primes of $\Q$, be a family of compact, open subgroups of $\GSp(4,\Q_\ell)$ such that $K_\ell=\GSp(4,\Z_\ell)$
for almost all $\ell$ and $\lambda (K_\ell) = \Z_\ell^\times$ for all $\ell$.  To $\mathcal{F}$ and $k$ we will associate a space of automorphic forms on $\GSp(4,\A)$ and a space of Siegel modular forms of degree two.  Set 
$$
K_{\mathcal{F}} = \prod_{\ell<\infty} K_\ell \subset \GSp(4,\A_f).
$$
Since $\lambda (K_\ell) = \Z_\ell^\times$ for all finite $\ell$, 
strong approximation for $\SSp(4)$ implies that 
$$
\GSp(4,\A) = \GSp(4,\Q) \GSp(4,\R)^+ K_{\mathcal{F}}.
$$
 Let $\bm \chi: \A^\times \to \C^\times$ be a
quadratic Hecke character, and let $\bm \chi=\prod_{\ell \leq \infty} \bm\chi_\ell$ be the decomposition of $\bm \chi$ as a product of local characters.  Then $\bm \chi_\infty (\R^\times_{>0}) =1$. Also let 
 $$
 K_\infty=\left\{\begin{bmatrix}A&B\\-B&A\end{bmatrix}\in\GL(4,\R):\,{}^t\!AA+{}^t\!BB={\bf1}, \,{}^t\!AB={}^t\!BA\right\}.
 $$ 
We define $S_k(K_{\mathcal{F}},\bm\chi)$ to be the space of continuous functions $\Phi: \GSp(4,\A) \to \C$
such that
\begin{enumerate}
\item  $\Phi ( \rho g ) = \Phi (g)$ for all $\rho \in \GSp(4,\Q)$ and $g \in \GSp(4,\A)$;
\item $\Phi (g z) = \Phi (g)$ for all $z \in \A^\times$ and $g \in \GSp(4,\A)$;
\item $\Phi (g\kappa_\ell) = \bm \chi_\ell(\lambda(\kappa_\ell)) \Phi (g)$ for all $\kappa_\ell \in K_\ell$, $g \in \GSp(4,\A)$ and finite primes $\ell$ of $\Q$;
\item $\Phi (gk_\infty) = j(k_\infty,I)^{-k}\Phi  (g)$ for all $k_\infty \in K_\infty$ and $g \in \GSp(4,\A)$;
\item For any  proper parabolic subgroup $P$ of $\GSp(4)$
$$
\int\limits_{N_P(\Q)\backslash N_P(\A)} \Phi(ng)\dif n =0
$$
for all $g \in \GSp(4,\A)$; here $N_P$ is 
the unipotent radical of $P$;
\item For any $g_f \in \GSp(4,\A_f)$, the function $\GSp(4,\R)^+ \to \C$ defined by $g_\infty \mapsto \Phi(g_fg_\infty)$
is smooth and is annihilated by $\mathfrak{\p_\C^-}$, where we refer to Section 3.5 of \cite{AS} for the definition of $\mathfrak{\p_\C^-}$. 
\end{enumerate}
On the other hand, to $\mathcal{F}$ we
associate a subgroup of $\SSp(4,\Q)$ that is commensurable with
$\GSp(4,\Z)$,
$$
\Gamma_{\mathcal{F}} = \GSp(4,\Q) \cap \GSp(4,\R)^+ \prod_{\ell<\infty} K_\ell.
$$
We define $S_k(\Gamma_{\mathcal{F}})$ as in Section \ref{notationsec}.
\begin{lemma}\label{KFisolemma}
Let $\bm \chi$, $\mathcal{F}$ and $k$ be as above.  For $F\in S_k(\Gamma_\mathcal{F})$, define $\Phi_F:\GSp(4,\A)\rightarrow\C$ by
$$
\Phi_F ( \rho h \kappa) = \lambda(h)^k j(h,I)^{-k} \cdot \prod_{\ell< \infty} \bm \chi_\ell(\lambda(\kappa_\ell)) \cdot F(h\langle I \rangle )  
$$
for $\rho \in \GSp(4,\Q),  h \in \GSp(4,\R)^+, \kappa \in K_{\mathcal{F}}=\prod_{\ell<\infty} K_\ell$. 
Then $\Phi_F$ is a well-defined element of $S_k(K_{\mathcal{F}}, \bm \chi)$, so that there is a complex linear map
\begin{equation}
\label{modtoautoeq}
S_k(\Gamma_{\mathcal{F}}) \longrightarrow S_k(K_{\mathcal{F}}, \bm \chi). 
\end{equation}
Conversely, let $\Phi \in S_k(K_{\mathcal{F}}, \bm \chi)$. Define $F_\Phi: \mathfrak{H}_2 \to \C$ by
$$
F_\Phi (Z) = \lambda(h)^{-k} j(h,I)^k \Phi(h)
$$
for $Z \in \mathfrak{H}_2$ with $h \in \GSp(4,\R)^+$ with $h\langle I \rangle = Z$. Then
$F_\Phi$ is well-defined and contained in $S_k(\Gamma_{\mathcal{F}})$, so that there is complex linear map
\begin{equation}
\label{autotomodeq}
S_k(K_{\mathcal{F}}, \bm \chi)  \longrightarrow S_k(\Gamma_{\mathcal{F}}). 
\end{equation}
Moreover, the maps \eqref{modtoautoeq} and \eqref{autotomodeq} are inverses of each other. 
\end{lemma}
\begin{proof}
To prove that $\Phi_F$ is well defined, suppose that $\rho h \kappa=\rho'h'\kappa'$ for $\rho, \rho'\in\GSp(4,\Q)$, $h, h'\in\GSp(4,\R)^+$, and $\kappa, \kappa'\in\prod_{\ell<\infty} K_\ell$.  Comparing components, we have that $\rho h=\rho'h'\in \GSp(4, \R)^+$ and $\rho\kappa_\ell=\rho'\kappa'_\ell\in\GSp(4,\Q_\ell)$ for $\ell<\infty$.  Therefore $\rho_0=\rho'{}^{-1}\rho\in\Gamma_\mathcal{F}$.  Noting that $\lambda(\rho_0)=1$, we have that
\begin{align*}
\lambda(h')^kj(h',I)^{-k}&\prod_{\ell<\infty}\bm \chi_\ell(\lambda(\kappa'_\ell))F(h'\langle I\rangle)\\
=&\lambda(\rho_0h)^kj(\rho_0h,I)^{-k}\prod_{\ell<\infty}\bm \chi_\ell(\lambda(\rho_0\kappa_\ell))F((\rho_0h)\langle I\rangle)\\
=&\lambda(\rho_0)^k\lambda(h)^kj(\rho_0,h\langle I\rangle)^{-k}j(h,I)^{-k}\prod_{\ell<\infty}\bm \chi_\ell(\lambda(\kappa_\ell))F(\rho_0\langle h\langle I\rangle\rangle)\\
=&\lambda(h)^kj(h,I)^{-k}\prod_{\ell<\infty}\bm \chi_\ell(\lambda(\kappa_\ell))(F|_k\rho_0)( h\langle I\rangle)\\
=&\lambda(h)^kj(h,I)^{-k}\prod_{\ell<\infty}\bm \chi_\ell(\lambda(\kappa_\ell))F( h\langle I\rangle).\\
\end{align*}
This shows that $\Phi_F$ is well-defined.  Straightforward calculations show that the function also satisfies first four conditions in the definition of $S_k(K_\mathcal{F},\bm \chi)$.  The proofs of the fifth and sixth conditions are similar to the proofs of Lemma 5 and Lemma 7, respectively, in \cite{AS}.  Similar calculations show that $F_\Phi\in S_k(\Gamma_\mathcal{F})$, for $\Phi$ in $S_k(K_\mathcal{F},\bm \chi)$.  Finally, it is clear that these two maps are inverses of each other.  
\end{proof}
\begin{lemma}\label{operatorlemma}
Let $k$ be a positive integer and let $\mathcal{F}_1=\{K^1_\ell\}_\ell$ and $\mathcal{F}_2=\{K^2_\ell\}_\ell$, where $\ell$ runs over the finite primes of $\Q$, be families of compact, open subgroups of $\GSp(4,\Q_\ell)$ such that $K^1_\ell=K^2_\ell=\GSp(4,\Z_\ell)$ for almost all $\ell$ and $\lambda(K^1_\ell)=\lambda(K^2_\ell)=\Z_\ell^\times$ for all $\ell$. Let $\bm \chi_1$ and $\bm \chi_2$ be quadratic Hecke characters. Suppose that there is a linear map
$$
T:S_k(K_{\mathcal{F}_1},\bm \chi_1)\rightarrow S_k(K_{\mathcal{F}_2},\bm \chi_2)
$$
given by a right translation formula at the $p$th place,
$$
T(\Phi_1)=\sum_{i=1}^tc_iR\big(B_{i,p}\big)\Phi_1,
$$
for $\Phi_1\in S_k(K_{\mathcal{F}_1},\bm \chi_1)$. Here $c_i\in\C^\times$ and $B_i\in\SSp(4,\Q)^+$ are such that $B_i\in K^1_\ell$ for $\ell\neq p$, $i\in\{1,...,t\}$. Then the composition, $\mathcal{T}$,
\begin{equation}
\label{TmathcalTeq}
\begin{CD}
S_k(\Gamma_{\mathcal{F}_1})@>\mathcal{T}>>S_k(\Gamma_{\mathcal{F}_2})\\
@V\wr VV @AA\wr A\\
S_k(K_{\mathcal{F}_1},\bm \chi_1)@>T>>S_k(K_{\mathcal{F}_2},\bm \chi_2)
\end{CD}
\end{equation}
is given by the formula
$$
\mathcal{T}(F)= \sum_{i=1}^t c_i \cdot F|_k (B_i)^{-1}
$$
for $F \in S_k(\Gamma_{\mathcal{F}_1})$.
\end{lemma}
\begin{proof}
Let $F \in S_k(\Gamma_{\mathcal{F}_1})$. By the isomorphism \eqref{modtoautoeq} for the family $\mathcal{F}_1$ with character $\bm \chi_1$ we have $\Phi_1=\Phi_F \in S_k(K_{\mathcal{F}_1}, \bm \chi_1)$. 
Using the isomorphism \eqref{autotomodeq} for the family $\mathcal{F}_2$ with character $\bm \chi_2$, we calculate the composition $F_{T(\Phi_1)}$. Let $Z\in\mathfrak{H}_2$ and let $h\in\GSp(4,\R)^+$ be such that $h\langle I\rangle=Z$.  Then, using that $\lambda(B_i)=1$, we have that
\begin{align*}
F_{T(\Phi_1)}(Z)
&=\lambda(h)^{-k} j(h,I)^k \sum_{i=1}^t c_i \Phi_1 (h B_{i,p})\\
&=\lambda(h)^{-k} j(h,I)^k \sum_{i=1}^t c_i \Phi_1 (B_i^{-1} h B_{i,p})\\
&=\lambda(h)^{-k} j(h,I)^k \sum_{i=1}^t c_i \Phi_1 (B_{i,\infty}^{-1} h)\\
&=\lambda(h)^{-k} j(h,I)^k \sum_{i=1}^t c_i \lambda(B_i^{-1} h)^k j(B_i^{-1} h,I)^{-k} 
F ((B_i^{-1} h)\langle I \rangle  )\\
&= j(h,I)^k \sum_{i=1}^t c_i  j(B_i^{-1} , Z)^{-k} j( h,I)^{-k} 
F (  B_i^{-1} \langle Z \rangle  )\\
&= \sum_{i=1}^t c_i  j(B_i^{-1} , Z )^{-k} 
F (  B_i^{-1} \langle Z \rangle  )\\
&= \sum_{i=1}^t c_i (F|_k B_i^{-1})(Z).
\end{align*}
Hence,
$$
\mathcal{T}(F)=F_{T(\Phi_1)} =  \sum_{i=1}^t c_i\cdot F|_k B_i^{-1}.
$$
This completes the proof.
\end{proof}
\section{Proof of the Paramodular Twisting Theorem}
\label{maintheoremproofsubsec}

\begin{proof}[Proof of the \hyperlink{mainsiegeltheorem}{Paramodular Twisting Theorem}]
We will use Lemma \ref{operatorlemma}, Lemma \ref{heckeactionlemma}, and the \hyperlink{localtwisttheorem}{Local Twisting Theorem}.  Let $\bm \chi_1$ be the trivial Hecke character, and let ${\bm \chi_2}$ be the Hecke character corresponding to $\chi$, as in \eqref{chitochieq}. Let $\mathcal{F}_1=\{K^{\mathrm{para}}(\ell^{\val_\ell(N)}) \}_\ell$ and $\mathcal{F}_2=\{K^{\mathrm{para}}(\ell^{\val_\ell(Np^4)}) \}_\ell$.  Let $V$ be the $\C$ vector space of functions $\Phi:\GSp(4,\A) \to \C$ such that $\Phi \in V$ if and only if there exists a compact, open subgroup $\Gamma_1$ of $\GSp(4,\Q_p)$ such that $\Phi(gk)=\Phi(g)$ for $g \in \GSp(4,\A)$ and $k \in \Gamma_1$, and $\Phi(gz) = \Phi(g)$ for $g \in \GSp(4,\A)$ and $z \in \A^\times$. The group $\GSp(4,\Q_p)$ acts smoothly on $V$ by right translation, and for this action $\pi$ the center of $\GSp(4,\Q_p)$ acts trivially. Next, the operator $T_{\bm \chi_p}:V(0)\to V(4,\bm \chi_p)$ from the \hyperlink{localtwisttheorem}{Local Twisting Theorem} has the form 
$$
T_{\bm\chi_p}(\Phi) = \sum_{i=1}^t c_i \cdot \pi(B_i) \Phi
$$
where $c_i \in \C$ and $B_i \in K^{\mathrm{para}}(\ell^{\val_\ell(N)})$ for $\ell \neq p$, and we can arrange to have $B_i \in \SSp(4,\Q)$ for $i \in \{1,\dots,t\}$ by replacing $\OF$ by $\Z$, $\p^k$ by $p^k\Z$, $\varpi$ and $q$ by $p$. 
The space $S_k(K_{\mathcal{F}_1},\bm \chi_1)$ is contained in $V$, and moreover one verifies that the image of the restriction $T$ of $T_{\bm \chi_p}$ to $S_k(K_{\mathcal{F}_1},\bm \chi_1)$ is contained in $S_k(K_{\mathcal{F}_2},\bm \chi_2)$. By Lemma \ref{operatorlemma}, the map $\mathcal{T} : S_k(\Gamma^{\mathrm{para}}(N)) \to S_k(\Gamma^{\mathrm{para}}(Np^4))$ given by 
\eqref{TmathcalTeq} has the formula 
$$
\mathcal{T}(F)=  \sum_{i=1}^t c_i\cdot F|_k B_i^{-1}, \qquad F \in S_k(\Gamma^{\mathrm{para}}(N)).
$$
We let $\mathcal{T}_\chi=\mathcal{T}$; the formula in the statement of the theorem is a consequence of the expressions for $c_i$ and $B_i$ in  \hyperlink{localtwisttheorem}{Local Twisting Theorem}.  \end{proof}

\section{Proof of the $L$-function Theorem}\label{heckeoperators}
Throughout this section,
let $M$ and $k$ be positive integers, $\bm{\chi}$ a quadratic Hecke character, and $\ell$ a fixed rational prime.  We turn now to the paramodular family $\mathcal{F}=\{K^{\mathrm{para}}(\ell^{\val_\ell(M)}) \}_\ell$.  We will determine the explicit relationship between the Hecke operators for Siegel paramodular forms and those for twisted automorphic forms.  Let $\ell$ be a rational prime and define the Hecke operators $T(1,1,\ell,\ell)$ and $T(1,\ell,\ell,\ell^2)$ on $S_k(\Gamma^{\mathrm{para}}(M))$ by 
\begin{equation}
T(1,1,\ell,\ell)F=\ell^{k-3}\sum_i F|_ka_i, \qquad T(1,\ell,\ell,\ell^2)F=\ell^{2(k-3)}\sum_jF|_kb_j\label{heckeoperatordef}
\end{equation}
where
\begin{align}
\Gamma^{\mathrm{para}}(M)\begin{bmatrix}1&&&\\&1&&\\&&\ell&\\&&&\ell\end{bmatrix}\Gamma^{\mathrm{para}}(M)&=\bigsqcup\Gamma^{\mathrm{para}}(M)a_i \label{heckeelleq}\\
\intertext{and}
\Gamma^{\mathrm{para}}(M)\begin{bmatrix}1&&&\\&\ell&&\\&&\ell^2\\&&&\ell\end{bmatrix}\Gamma^{\mathrm{para}}(M)&=\bigsqcup\Gamma^{\mathrm{para}}(M)b_j\label{heckeellsquaredeq},
\end{align}
are disjoint decompositions.  We define the Atkin-Lehner involution $U_\ell$ on $S_k(\Gamma^{\mathrm{para}}(M))$ as follows.  Choose a matrix $\gamma_\ell\in\SSp(4,\Z)$ such that 
$$
\gamma_\ell\equiv\begin{bmatrix}&&&1\\&&1&\\&-1&&\\-1&&&\end{bmatrix}\bmod{\ell^{\val_\ell(M)}}\quad\text{and}\quad\gamma_\ell\equiv\begin{bmatrix}1&&&\\&1&&\\&&1&\\&&&1\end{bmatrix}\bmod{M\ell^{-\val_\ell(M)}}.
$$
Then, 
\begin{equation}\label{atkinlehnerdef}
U_\ell=\gamma_\ell\begin{bmatrix}\ell^{\val_\ell(M)}&&&\\&\ell^{\val_\ell(M)}&&\\&&1&\\&&&1\end{bmatrix}
\end{equation}
normalizes $\Gamma^{\mathrm{para}}(M)$ and $U_\ell^2$ is contained in $\ell^{\val_\ell(M)}\Gamma^{\mathrm{para}}(M)$, implying that $F\mapsto F|_kU_\ell$ is indeed an involution of $S_k(\Gamma^{\mathrm{para}}(M))$.  

\begin{lemma}\label{heckedecompositionlemma}
Let $M$ be a positive integer, let $\ell$ be a prime and set $r=\mathrm{val}_\ell(M)$.  There exist finite disjoint decompositions
\begin{align*}
&K^{\mathrm{para}}(\ell^r)\begin{bmatrix}\ell&&&\\&\ell&&\\&&1&\\&&&1\end{bmatrix}K^{\mathrm{para}}(\ell^r)=\bigsqcup g_iK^{\mathrm{para}}(\ell^r)\\\intertext{and}
&K^{\mathrm{para}}(\ell^r)\begin{bmatrix}\ell^2&&&\\&\ell&&\\&&1&\\&&&\ell\end{bmatrix}K^{\mathrm{para}}(\ell^r)=\bigsqcup h_jK^{\mathrm{para}}(\ell^r)\\\intertext{such that}
&\Gamma^\mathrm{para}(M)\begin{bmatrix}1&&&\\&1&&\\&&\ell&\\&&&\ell\end{bmatrix}\Gamma^\mathrm{para}(M)=\bigsqcup\Gamma^\mathrm{para}(M)\ell g_i^{-1}\\\intertext{and}
&\Gamma^\mathrm{para}(M)\begin{bmatrix}1&&&\\&\ell&&\\&&\ell^2&\\&&&\ell\end{bmatrix}\Gamma^\mathrm{para}(M)=\bigsqcup\Gamma^\mathrm{para}(M)\ell^2 {h_j}^{-1}.
\end{align*}
The explicit forms of the representatives are given in the proof.
\end{lemma}
\begin{proof}
The explicit form of the decompositions given in Section 6.1 of \cite{RS} implies that we can choose representatives $g_i$ and $h_j$ such that 
\begin{align*}
&K^{\mathrm{para}}(\ell^r)\begin{bmatrix}\ell&&&\\&\ell&&\\&&1&\\&&&1\end{bmatrix}K^{\mathrm{para}}(\ell^r)=\bigsqcup_{i=1}^Dg_iK^{\mathrm{para}}(\ell^r),\\
&K^{\mathrm{para}}(\ell^r)\begin{bmatrix}\ell^2&&&\\&\ell&&\\&&1&\\&&&\ell\end{bmatrix}K^{\mathrm{para}}(\ell^r)=\bigsqcup_{j=1}^{D'} h_jK^{\mathrm{para}}(\ell^r),\\
\intertext{and}
&\ell g^{-1}_i\in\Gamma^{\mathrm{para}}(M)\begin{bmatrix}1&&&\\&1&&\\&&\ell&\\&&&\ell\end{bmatrix}\Gamma^{\mathrm{para}}(M),\quad
\ell^2h_j^{-1}\in\Gamma^{\mathrm{para}}(M)\begin{bmatrix}1&&&\\&\ell&&\\&&\ell^2&\\&&&\ell\end{bmatrix}\Gamma^{\mathrm{para}}(M),
\end{align*}
for $i\in\{1,\dots, D\}$ and $j\in\{1,\dots,D'\}$.  For this, it is useful to note that $\Gamma^{\mathrm{para}}(M)$ contains several symmetry elements: 
$$
\begin{bmatrix}1&&&\\&&&1\\&&1&\\&-1&&&\end{bmatrix},\qquad\begin{bmatrix}&&-M^{-1}&\\&1&&\\M&&&\\&&&1\end{bmatrix},
$$  and in the case that $r=0$,  the element
$\left[\begin{smallmatrix}A&\\&{}^tA^{-1}\end{smallmatrix}\right]$
where $A\in\Gamma_0(M)\subset\SL(2,\Z)$ and $A\left[\begin{smallmatrix}1&\\&\ell\end{smallmatrix}\right]A^{-1}=\left[\begin{smallmatrix}\ell&\\&1\end{smallmatrix}\right]$.  It follows that the cosets $\Gamma^{\mathrm{para}}(M)\ell g_i^{-1}$ are mutually disjoint and contained in the first double coset, and the cosets $\Gamma^{\mathrm{para}}(M)\ell^2h_j^{-1}$ are mutually disjoint and contained in the second double coset.  It remains to prove that the number of disjoint cosets in the first and second double cosets are $D$ and $D'$, respectively.  Suppose that 
$$
\Gamma^{\mathrm{para}}(M)\begin{bmatrix}1&&&\\&1&&\\&&\ell&\\&&&\ell\end{bmatrix}\Gamma^{\mathrm{para}}(M)=\bigsqcup_{i=1}^d\Gamma^{\mathrm{para}}(M)g'_i
$$
is a disjoint decomposition.  We have already shown that $d\geq D$;  we need to prove that $D\geq d$.  We have
$$
K^{\mathrm{para}}(\ell^r)\begin{bmatrix}1&&&\\&1&&\\&&\ell&\\&&&\ell\end{bmatrix}K^{\mathrm{para}}(\ell^r)\supset K^{\mathrm{para}}(\ell^r)\begin{bmatrix}1&&&\\&1&&\\&&\ell&\\&&&\ell\end{bmatrix}\Gamma^{\mathrm{para}}(M)=\bigcup_{i=1}^dK^{\mathrm{para}}(\ell^r)g'_i.
$$
We claim that these cosets are disjoint.  For suppose $K^{\mathrm{para}}(\ell^r)g'_i=K^{\mathrm{para}}(\ell^r)g'_j$.  This implies that $g'_ig'_j{}^{-1}\in K^{\mathrm{para}}(\ell^r)$.  Since for any prime $q$ with $q\neq\ell$ we have $g'_ig'_j{}^{-1}\in K^{\mathrm{para}}(q^{\val_q(M)})$ as all the elements of 
$$
\Gamma^{\mathrm{para}}(M)\begin{bmatrix}1&&&\\&1&&\\&&\ell&\\&&&\ell\end{bmatrix}\Gamma^{\mathrm{para}}(M)
$$
are contained in $K^{\mathrm{para}}(q^{\val_q(M)})$, we have  $g'_ig'_j{}^{-1}\in\Gamma^{\mathrm{para}}(M)$.  This implies that $i=j$.  It follows that $D\geq d$.  The proof that $d'=D'$ is similar.
\end{proof}

Let $V$ be the $\C$ vector space of functions $\Phi:\GSp(4,\A) \to \C$ such that $\Phi \in V$ if and only if there exists a compact, open subgroup $\Gamma_1$ of $\GSp(4,\Q_\ell)$ such that $\Phi(gk)=\Phi(g)$ for $g \in \GSp(4,\A)$ and $k \in \Gamma_1$, and $\Phi(gz) = \Phi(g)$ for $g \in \GSp(4,\A)$ and $z \in \A^\times$. The group $\GSp(4,\Q_\ell)$ acts smoothly on $V$ by right translation, and for this action, denoted by $\pi$, the center of $\GSp(4,\Q_\ell)$ acts trivially.  Assume that $\bm\chi_\ell$ is unramified.  Then $S_k(K_{\mathcal{F}},\bm\chi)\subset V(\val_\ell(M))$, where the last space consists of the vectors in $V$ that are fixed by $K^{\mathrm{para}}(\ell^{\val_\ell(M)})$ as in \cite{RS}.  Let $T_{1,0}$ and $T_{0,1}$ be the local Hecke operators and $u_{\val_{\ell}(M)}$ be the local Atkin-Lehner operator acting on $V(\val_\ell(M))$ as in \cite{RS}.  These operators preserve the subspace $S_k(K_{\mathcal{F}},\bm\chi)\subset V(\val_\ell(M))$.

\begin{lemma}\label{heckeactionlemma}
Assume that $\bm\chi_\ell$ is unramified. Let $\Phi\in S_k(K_\mathcal{F},\bm{\chi})$ and let $F_\Phi\in S_k(\Gamma^{\mathrm{para}}(M))$ be the Siegel modular form corresponding to $\Phi$ under the isomorphism \eqref{autotomodeq} of Lemma \ref{KFisolemma}. Then,
\begin{enumerate}
\item $T(1,1,\ell,\ell)F_\Phi=\ell^{k-3}\bm\chi_{\ell}(\ell)F_{T_{1,0}\Phi}$,
\item $T(1,\ell,\ell,\ell^2)F_\Phi=\ell^{2(k-3)}F_{T_{0,1}\Phi}$,
\item $F|_kU_\ell=\bm\chi_\ell(\ell)^{\val_\ell(M)}F_{u_{\val_\ell(M)}\Phi}.$
\end{enumerate}
\end{lemma}

\begin{proof}
Let $Z\in\mathfrak{H}_2$ and let $g_\infty\in\SSp(4,\R)$ such that $g_\infty\langle I\rangle=Z$. For the first assertion, we will use the coset representatives from Lemma \ref{heckedecompositionlemma}.  Note that from the explicit forms given in the proof, we have that $g_i\in K^{\mathrm{para}}(q^{\val_q(M)})$ for primes $q\neq\ell$ and $\lambda(g_i)=\ell$ for each $i$. We calculate
\begin{align*}
(T(1,1,\ell,\ell)F_\Phi)(Z)&=\ell^{k-3}\sum_i(F_\Phi|_k\ell g_i^{-1})(Z)\\
&=\ell^{k-3}\sum_i\lambda(\ell g_i^{-1})^kj(\ell g_i^{-1}, g_\infty\langle I\rangle)^{-k} F_\Phi(\ell g_i^{-1}g_\infty\langle I\rangle)\\
&=\ell^{k-3}\sum_i\lambda(\ell g_i^{-1})^kj(\ell g_i^{-1}, g_\infty\langle I\rangle)^{-k}\lambda(\ell g_i^{-1}g_\infty)^{-k}j(\ell g_i^{-1}g_\infty,I)^k\Phi((\ell g_i^{-1})_\infty g_\infty)\\
&=\ell^{k-3}\sum_i\lambda(g_\infty)^{-k}j(g_\infty,I)^k\Phi(g_{i,\infty}^{-1}\, g_\infty)\\
&=\ell^{k-3}\sum_i\lambda(g_\infty)^{-k}j(g_\infty,I)^k\Phi(g_{i,f}\, g_\infty)\\
&=\ell^{k-3}\sum_i\lambda(g_\infty)^{-k}j(g_\infty,I)^k\prod_{q\neq\ell}\bm\chi_q(\lambda(g_i))\Phi(g_\infty\, g_{i,\ell})\\
&=\ell^{k-3}\bm\chi_\ell(\ell)\sum_i\lambda(g_\infty)^{-k}j(g_\infty,I)^k\bm\Phi(g_\infty\, g_{i,\ell})\\
&=\ell^{k-3}\bm\chi_\ell(\ell)\lambda(g_\infty)^{-k}j(g_\infty,I)^k(T_{1,0}\Phi)(g_\infty)\\
&=\ell^{k-3}\bm\chi_{\ell}(\ell)(F_{T_{1,0}}\Phi)(Z).
\end{align*}
The proof of the second and third assertions are similar.
\end{proof}
\begin{proof}[Proof of  the \hyperlink{heckeoperatortheorem}{$L$-function Theorem}] To prove (1),
let $F\in S_k(\Gamma^{\mathrm{para}}(N))$  and let $\Phi\in S_k(K_{\mathcal{F}_1},\bm\chi_1)$ with $F=F_\Phi$.  Let $\ell$ be a rational prime with $\ell\neq p$ so that $\bm\chi_{1,\ell}$ and $\bm\chi_{2,\ell}$ are unramified.  Moreover, we have that $\bm\chi_{1,\ell}$ is trivial and  $\bm\chi_{2,\ell}(\ell)=\chi(\ell)$.  Then, using \eqref{TmathcalTeq}  and Lemma \ref{heckeactionlemma} we have
\begin{align*}
T(1,1,\ell,\ell)\mathcal{T}_\chi(F_\Phi)&=T(1,1,\ell,\ell)F_{T(\Phi)}\\
&=\ell^{k-3}\chi(\ell)F_{T_{1,0}(\ell)T(\Phi)}\\
&=\ell^{k-3}\chi(\ell)F_{T(T_{1,0}(\ell)\Phi)}\\
&=\chi(\ell)\mathcal{T}_\chi(\ell^{k-3}F_{T_{1,0}(\ell)\Phi})\\
&=\chi(\ell)\mathcal{T}_\chi(T(1,1,\ell,\ell)F_\Phi).
\end{align*}
The proofs for the other two operators are similar.

 To prove (2), we recall from the Paramodular Twisting Theorem that $\mathcal{T}_\chi(F)$  is in $S_k(\Gamma^{\mathrm{para}}(Np^4))$.  To show that $\mathcal{T}_\chi(F)$ is a newform, consider the automorphic form associated to $F$, $\Phi_{F}$, as defined in Lemma \ref{KFisolemma}.  Let $V$ be the subspace of cusp forms on $\GSp(4,\A)$ generated by $\Phi_F$ under the action of $\GSp(4,\A)$.  The space $V$ is a direct sum of finitely many non-zero irreducible $\GSp(4,\A)$-subspaces $V_1,\cdots,V_t$, so that $\Phi_F=\sum_{i=1}^t\Phi_i$ with $\Phi_i\in V_i$.  Let $V_i\simeq\otimes_{\ell\leq\infty}V_{i,\ell}$ where $V_i$ is an irreducible, admissible representation of $\GSp(4,\Q_\ell)$.
Let $F_{\Phi_i}\in S_k(\Gamma^\mathrm{para}(N))$ be the Siegel modular forms associated to the $\Phi_i$ by Lemma \ref{KFisolemma}; then these are also newforms.  In particular, $\Phi_i=\otimes_{\ell\leq\infty}\Phi_{i,\ell}$ is a pure tensor with $\Phi_{i,\ell}\in V_{i,\ell}$ a local newform for $\ell<\infty$.  As in the proof of the Paramodular Twisting Theorem, $\Phi_{\mathcal{T}_\chi(F)}={T_{\bm\chi_p}\Phi_F}$, and thus $\Phi_{\mathcal{T}_\chi(F)}=\sum_{i=1}^tT_{\bm\chi_p}\Phi_i=\sum_{i=1}^t(\otimes_{\ell\neq p}\Phi_{i,\ell})\otimes T_{\bm\chi_p}\Phi_{i,p}$.  If $T_{\bm\chi_p}\Phi_{i,p}\neq0$, then by the tables in Appendix A of \cite{RS}, $V_{i,p}$ is an unramified Type I representation, $\bm\chi_p\otimes V_{i,p}$ is paramodular of level $p^4$, and $T_{\bm\chi_p}\Phi_{i,p}$ is a local newform inside $\bm\chi_p\otimes V_{i,p}$.  From this and tables A.12 and A.14 in \cite{RS}, the eigenvalues for $\ell=p$ are as stated above.  In the case that $\ell\neq p$, the eigenvalues are determined by the commutation relations given in (1).  The final assertion then follows from the definition of the involved local $L$-factors.
\end{proof}

\section{Proof of the Local Twisting Theorem}
\label{localsec}
Throughout the remainder of the paper, we will use the following notation. Let $F$ be a nonarchimedean local field of characteristic zero and odd residual characteristic with ring of integers $\OF$ and generator $\varpi$ of the maximal ideal $\p$ of $\OF$. We let $q$ be the number of elements of $\OF/\p$ and use the absolute value on $F$ such that $|\varpi|=q^{-1}$. We use the Haar measure on the additive group $F$ that assigns $\OF$ measure $1$ and the Haar measure on the multiplicative group $F^\times$ that assigns $\OF^\times$ measure $1-q^{-1}$. We let
$\chi$ be a quadratic character of $F^\times$ of conductor $\p$. Let 
$$
J'=\begin{bmatrix}&&&1 \\ &&1& \\ &-1&& \\ -1&&& \end{bmatrix}.
$$
For the next two sections, we define  $\GSp(4,F)$ as the subgroup of all $g\in\GL(4,F)$ such that ${}^tgJ'g=\lambda(g)J'$ for some $\lambda(g)\in F^\times$ called the multiplier of $g$. For $n$ a non-negative integer, we let $\mathrm{K}(\p^n)$ be the subgroup of $k \in \GSp(4,F)$ such that $\lambda(k) \in \OF^\times$ and
$$
k \in \begin{bmatrix} \OF&\OF&\OF&\p^{-n} \\ \p^n & \OF & \OF & \OF \\ \p^n & \OF & \OF & \OF \\ \p^n & \p^n & \p^n & \OF\end{bmatrix}.
$$
Let $(\pi,V)$ be a smooth representation of $\GSp(4,F)$ for which the center acts trivially.  If $n$ is a non-negative integer, then $V(n)$ is the subspace of vectors fixed by the paramodular subgroup $\mathrm{K}(\p^n)$; also, we let $V(n,\chi)$ be the subspace of vectors $v \in V$ such that $\pi(k)v = \chi(\lambda(k)) v$ for $k \in \mathrm{K}(\p^n)$.  Finally, let 
$$
\eta=\begin{bmatrix} \varpi^{-1} &&& \\ &1&& \\ &&1& \\ &&& \varpi \end{bmatrix} ,\qquad
\tau = \begin{bmatrix} 1&&& \\ &\varpi^{-1}&& \\ &&\varpi& \\ &&&1 \end{bmatrix},\qquad t_4=\begin{bmatrix}&&&-\varpi^{-4}\\&1&&\\&&1&\\\varpi^4&&& \end{bmatrix}.
$$
Usually, we will write $\eta$ and $\tau$ for $\pi(\eta)$ and $\pi(\tau)$, respectively.

In \cite{JR1} we constructed a twisting map,
\begin{equation}\label{localtwistmapeq}
T_\chi:V(0)\rightarrow V(4,\chi),
\end{equation}
given by
\begin{align}
T&_\chi(v) = 
q^3\int\limits_{\OF}\int\limits_{\OF}\int\limits_{\OF^\times}\int\limits_{\OF^\times}
\chi(ab)\pi(
\begin{bmatrix} 1&&& \\ &1&& \\ &x&1& \\ &&&1 \end{bmatrix} 
\begin{bmatrix} 1&-a\varpi^{-1} & b\varpi^{-2} & z\varpi^{-4} \\ &1&&b\varpi^{-2} \\ &&1&a\varpi^{-1} \\ &&&1 \end{bmatrix} )
\tau v\dif a\dif b\dif x\dif z\tag{P1}\label{termoneeq} \\
&+q^3\int\limits_{\OF}\int\limits_{\p}\int\limits_{\OF^\times}\int\limits_{\OF^\times}
\chi(ab)\pi(\!
\begin{bmatrix} 1&&& \\ &&1& \\ &-1&& \\ &&&1 \end{bmatrix}\!\!\!
\begin{bmatrix} 1&&& \\ &1&& \\ &y&1& \\ &&&1 \end{bmatrix}\!\!\!  
\begin{bmatrix} 1&-a\varpi^{-1} & b\varpi^{-2} & z\varpi^{-4} \\ &1&&b\varpi^{-2} \\ &&1&a\varpi^{-1} \\ &&&1 \end{bmatrix}\! )
\tau v\dif a\dif b\dif y\dif z\tag{P2}\label{termtwoeq} \\
&+q^2\!\int\limits_{\OF}\int\limits_{\OF}\int\limits_{\OF^\times}\int\limits_{\OF^\times} \chi(ab)
\pi(t_4
\begin{bmatrix} 1&&& \\ &1&& \\ &x&1& \\ &&&1 \end{bmatrix} 
\begin{bmatrix} 1&-a\varpi^{-1} & b\varpi^{-2} & z\varpi^{-3} \\ &1&&b\varpi^{-2} \\ &&1&a\varpi^{-1} \\ &&&1 \end{bmatrix} )
\tau v\dif a\dif b\dif x\dif z\tag{P3} \label{termthreeeq} \\
&+q^2\!\!\int\limits_{\OF}\!\int\limits_{\p}\!\int\limits_{\OF^\times} \int\limits_{\OF^\times}\chi(ab)
\pi(\!t_4
\begin{bmatrix} 1&&& \\ &&1& \\ &-1&& \\ &&&1 \end{bmatrix} \!\!\!
\begin{bmatrix} 1&&& \\ &1&& \\ &y&1& \\ &&&1 \end{bmatrix}\!\!\!
\begin{bmatrix} 1&-a\varpi^{-1} & b\varpi^{-2} & z\varpi^{-3} \\ &1&&b\varpi^{-2} \\ &&1&a\varpi^{-1} \\ &&&1 \end{bmatrix} )
\tau v\dif a\dif b\dif y\dif z. \tag{P4}\label{termfoureq}
\end{align}
 The \hyperlink{localtwisttheorem}{Local Twisting Theorem} asserts that the twisting map $T_\chi$ may be expressed in terms of upper triangular matrices.  Indeed, by the Iwasawa decomposition $\GSp(4,F)=B\cdot\GSp(4,\OF)$ where $B$ is the Borel subgroup of upper-triangular matrices in $\GSp(4,F)$.  Hence, if  $v\in V(0)$, so that $v$ is invariant under $\GSp(4,\OF)$, then it is possible to obtain a formula for $T_\chi(v)$ involving only upper-triangular matrices.  However, the Iwasawa decomposition does not give an algorithm for finding this decomposition, nor are these decompositions necessarily universal for a family of elements of $\GSp(4,F)$ depending on a parameter.

To attack this problem we use the following strategies.  In some cases, we directly provide an Iwasawa identity $g=bk$ where $g\in\GSp(4,F)$, $b\in B$, and $k\in \GSp(4,\OF)$.  In many cases, we are able to obtain an appropriate Iwasawa identity by using the following formal matrix identity
\begin{equation}\label{flipup}
\begin{bmatrix}1&\\x&1\end{bmatrix}=\begin{bmatrix}1&x^{-1}\\&1\end{bmatrix}\begin{bmatrix}-x^{-1}&\\&-x\end{bmatrix}\begin{bmatrix}&1\\-1&\end{bmatrix}\begin{bmatrix}1&x^{-1}\\&1\end{bmatrix}.
\end{equation}
Both methods require that we decompose the domains of integration in an advantageous manner and make appropriate changes of variables.  The assumptions on the character, $\chi\neq1$, $\chi^2=1$ and $\chi(1+\p)=1$, also play a significant role in the computations.

The following proof of the \hyperlink{localtwisttheorem}{Local Twisting Theorem} assembles calculations of the four terms of $T_\chi(v)$. The terms themselves are calculated in a series of lemmas which appear after the proof.  The proofs of the lemmas then appear in Section \ref{lemmaproofsec}.  The third term is by far the most delicate, and its calculation comprises the bulk of the remaining pages of this manuscript.  
\begin{proof}[Proof of the \hyperlink{localtwisttheorem}{Local Twisting Theorem}]
Substituting the formula for \eqref{termoneeq} from Lemma \ref{termonelemma}, \eqref{termtwoeq} from Lemma \ref{termtwolemma}, \eqref{termthreeeq} from Lemma \ref{termthreelemma}, and \eqref{termfoureq} from Lemma \ref{termfourlemma}, we have that the twisting operator is given by the formula
\begin{align*}
&T_\chi(v)=\\
&q^2
\int\limits_{\OF} \int\limits_{\OF^\times}\int\limits_{\OF^\times}\int\limits_{\OF^\times}\chi(ab)   \pi(
\begin{bmatrix} 1 & -(a+xb)\varpi^{-1}  & & \\ &1&& \\ &&1&(a+xb) \varpi^{-1} \\ &&&1 \end{bmatrix} 
\begin{bmatrix} 1 &   & b\varpi^{-2} & z \varpi^{-4} \\ &1&x^{-1}\varpi^{-1}&b\varpi^{-2} \\ &&1& \\ &&&1 \end{bmatrix} ) v \dif a \dif b \dif x \dif z\\
&+q\chi(-1)\eta\!\int\limits_{\OF^\times-(1+\p)}\int\limits_{\OF^\times}\int\limits_{\OF^\times} \int\limits_{\OF^\times - A(z)}\chi(abx)
\pi(\begin{bmatrix}
1&b \varpi^{-1} &&\\
&1&&\\
&&1&-b\varpi^{-1} \\
&&&1
\end{bmatrix}\\
&\begin{bmatrix}
1&&a\varpi^{-2} & -abx^{-1}(1+x-z)\varpi^{-3} \\
&1& -ab^{-1}xz (1-z +zx)^{-1}\varpi^{-1} & a\varpi^{-2} \\
&&1&\\
&&&1
\end{bmatrix}
) v\dif x\dif a\dif b\dif z\\
&+\chi(-1)\eta\!\int\limits_{\OF^\times-(1+\p)}\int\limits_{\OF^\times}\int\limits_{\OF^\times} \chi(b(1-z))
\pi(\begin{bmatrix}
1&&a\varpi^{-2}&-b\varpi^{-3}\\
&1&-a^2b^{-1}z\varpi^{-1}&a\varpi^{-2}\\
&&1&\\
&&&1
\end{bmatrix})v\dif a\dif b\dif z\\
&+q\eta\int\limits_{\OF}\int\limits_{\OF^\times}\int\limits_{\OF^\times} \chi(b)
\pi(\begin{bmatrix} 1& x\varpi^{-2}& & \\ &1&&\\ &&1&-x\varpi^{-2}\\ & &&1 \end{bmatrix}\begin{bmatrix} 1& & a\varpi^{-2}&(b\varpi-ax)\varpi^{-4} \\ &1&&a\varpi^{-2} \\ &&1&\\& &&1 \end{bmatrix}  )
v\dif a\dif b\dif x\\
&+\eta\int\limits_{\OF}\!\int\limits_{\OF^\times} \int\limits_{\OF^\times}\chi(ab)
\pi(\begin{bmatrix}
1&-a\varpi^{-2}&&\\
&1&&\\
&&1&a\varpi^{-2}\\
&&&1
\end{bmatrix}\begin{bmatrix}
1&&y\varpi^{-1}&a(y+b)\varpi^{-3}\\
&1&&y\varpi^{-1} \\
&&1&\\
&&&1
\end{bmatrix}
) v\dif a\dif b\dif y\\
&+q^{-1}\chi(-1)\eta^2\int\limits_{\OF^\times}\int\limits_{\OF^\times}\int\limits_{\OF^\times} \chi(bx)
\pi(
\begin{bmatrix}
1&&a\varpi^{-2}&  b(1+x\varpi)\varpi^{-2}\\
&1&a^2b^{-1}\varpi^{-2}&a\varpi^{-2}\\
&&1&\\
&&&1
\end{bmatrix} )
v\dif a\dif b\dif x\\
&+q\tau\int\limits_{\OF} \int\limits_{\OF^\times}\int\limits_{\OF^\times}\chi(ab)  
\pi( \begin{bmatrix} 1 & -a\varpi^{-2} & & \\ &1&& \\ &&1&a\varpi^{-2} \\ &&&1 \end{bmatrix} \begin{bmatrix} 1 &  & b \varpi^{-1} & z\varpi^{-4}\\ &1&&b\varpi^{-1} \\ &&1&  \\ &&&1 \end{bmatrix}  ) v\dif a\dif b\dif z\\
&+\eta\tau\int\limits_{\OF^\times-(1+\p)}\int\limits_{\OF^\times}\int\limits_{\OF^\times} \chi(abz(1-z))
\pi(\setlength{\arraycolsep}{3pt}\begin{bmatrix}
1&b\varpi^{-2}&\\
&1&&\\
 &&1&-b\varpi^{-2}\\
& &&1
\end{bmatrix}
\begin{bmatrix}
1 &  & a\varpi^{-1}&-ab(1-z)\varpi^{-3} \\
&1& & a\varpi^{-1}\\
&&1 & \\
&&&1
\end{bmatrix}
) v\dif a\dif b\dif z\\
&+q^{-2}\chi(-1)\eta^2\tau\int\limits_{\OF^\times}\int\limits_{\OF^\times}\int\limits_{\OF^\times} \chi(b)
\pi(
\begin{bmatrix}
1&a\varpi^{-1}& &\\
&1&&\\
 &&1&-a\varpi^{-1}\\
& &&1
\end{bmatrix}
\begin{bmatrix}
1&&&b\varpi^{-1} \\
 & 1& x\varpi^{-1}& \\
&&1& \\
&& & 1 
\end{bmatrix})
v\dif a\dif b\dif x\\
&+q^{-3}\chi(-1)\eta^2\tau^2\int\limits_{\OF^\times}\int\limits_{\OF^\times} \chi(b)
\pi(\begin{bmatrix}
1&a\varpi^{-2}& &b\varpi^{-1}\\
&1&&\\
 &&1&-a\varpi^{-2}\\
& &&1
\end{bmatrix} )
v\dif a\dif b\\
&+q^3\tau^{-1}\!\!
\int\limits_{\OF}\! \int\limits_{\OF^\times}\!\int\limits_{\OF^\times}\!\int\limits_{\OF^\times}\!\chi(ab)   \pi(
\setlength{\arraycolsep}{1pt}\begin{bmatrix} 1 & b\varpi^{-1}  & &  \\ &1&&\\ &&1&-b\varpi^{-1} \\ &&&1 \end{bmatrix} \!\!\!
\setlength{\arraycolsep}{1pt}\begin{bmatrix} 1 &  & -(b-xa\varpi)x\varpi^{-3} & z\varpi^{-4} \\ &1&x\varpi^{-2}&-(b-xa\varpi)x\varpi^{-3}\\ &&1& \\ &&&1 \end{bmatrix}  ) v\dif a\dif b\dif x\dif z\\
&+q^2\tau^{-1} \int\limits_{\OF} \int\limits_{\OF}\int\limits_{\OF^\times}\int\limits_{\OF^\times} \chi(ab) \pi(\setlength{\arraycolsep}{2pt}\begin{bmatrix} 1 &   b \varpi^{-1} && \\  &1&& \\ &&1&-b\varpi^{-1} \\  &&&1 \end{bmatrix} 
\setlength{\arraycolsep}{2pt}\begin{bmatrix} 1 &   &(a+by)\varpi^{-2} & z\varpi^{-4} \\  &1&-y\varpi^{-1}&(a+by)\varpi^{-2} \\ &&1&\\  &&&1 \end{bmatrix}  ) v \dif a\dif b\dif y\dif z\\
&+q\eta\tau^{-1}\int\limits_{\OF^\times-(1+\p)}\!\int\limits_{\OF}\!\int\limits_{\OF^\times} \int\limits_{\OF^\times}\chi(ab)
\pi(
\begin{bmatrix}
1&-a\varpi^{-1}&&\\
&1&&\\
&&1&a\varpi^{-1}\\
&&&1
\end{bmatrix}\\&
\begin{bmatrix}
1&&y\varpi^{-2}&a(y+bz)\varpi^{-3}\\
&1&a^{-1}(y+bz(z-1)^{-1})\varpi^{-1}&y\varpi^{-2} \\
&&1&\\
&&&1
\end{bmatrix}
) v\dif a\dif b\dif y\dif z\\
&+q^2\chi(-1)\eta\tau^{-1}\!\int\limits_{\OF^\times-(1+\p)}\int\limits_{\OF^\times}\int\limits_{\OF^\times} \int\limits_{\OF^\times}\chi(ba(1-z)z)
\pi(\begin{bmatrix}
1&b\varpi^{-1} &&\\
&1&&\\
&&1&-b\varpi^{-1} \\
&&&1
\end{bmatrix}\\
&\begin{bmatrix}
1&&-x\varpi^{-3}& b(x-za\varpi)\varpi^{-4} \\
&1&b^{-1}(x+a\varpi)\varpi^{-2}&-x\varpi^{-3}\\
&&1&\\
&&&1
\end{bmatrix})v\dif x\dif a\dif b\dif z\\
&+\chi(-1)\eta^2\tau^{-1}\int\limits_{\OF^\times}\int\limits_{\OF^\times}\int\limits_{\OF^\times} \chi(bx)
\pi(\begin{bmatrix}
1&&a\varpi^{-2}&b(x-1)\varpi^{-1}\\
&1&-a^2b^{-1}\varpi^{-3}&a\varpi^{-2}\\
&&1&\\
 & & &1 
\end{bmatrix} )
v\dif a\dif b\dif x\\
&+q\chi(-1)\eta^2\tau^{-2}\int\limits_{\OF^\times}\int\limits_{\OF^\times}\int\limits_{\OF^\times} \chi(b)
\pi(
\begin{bmatrix}
1&&a\varpi^{-2}&b\varpi^{-1}\\
&1&x \varpi^{-4}&a\varpi^{-2}\\
&&1&\\
 & & &1 
\end{bmatrix} )
v\dif a\dif b\dif x\\
&+\eta^2\tau^{-2}\int\limits_{\OF}\!\int\limits_{\OF^\times} \int\limits_{\OF^\times}\chi(a)
\pi(\begin{bmatrix}
1&&b\varpi^{-2}&-a\varpi^{-1}\\
&1&y\varpi^{-3}&b\varpi^{-2} \\
&&1&\\
&&&1
\end{bmatrix}
) v\dif a\dif b\dif y.
\end{align*}
For the remainder of the proof, we will simplify by combining pairs of terms and rewriting certain domains.  First we combine the terms involving $\eta^2\tau^{-2}$:
\begin{align*}
&q\chi(-1)\eta^2\tau^{-2}\int\limits_{\OF^\times}\int\limits_{\OF^\times}\int\limits_{\OF^\times} \chi(b)
\pi(
\begin{bmatrix}
1&&a\varpi^{-2}&b\varpi^{-1}\\
&1&x \varpi^{-4}&a\varpi^{-2}\\
&&1&\\
 & & &1 
\end{bmatrix} )
v\dif a\dif b\dif x\\
&+\eta^2\tau^{-2}\int\limits_{\OF}\!\int\limits_{\OF^\times} \int\limits_{\OF^\times}\chi(a)
\pi(\begin{bmatrix}
1&&b\varpi^{-2}&-a\varpi^{-1}\\
&1&y\varpi^{-3}&b\varpi^{-2} \\
&&1&\\
&&&1
\end{bmatrix}
) v\dif a\dif b\dif y\\
&=q\chi(-1)\eta^2\tau^{-2}\int\limits_{\OF}\int\limits_{\OF^\times}\int\limits_{\OF^\times} \chi(b)
\pi(
\begin{bmatrix}
1&&a\varpi^{-2}&b\varpi^{-1}\\
&1&x \varpi^{-4}&a\varpi^{-2}\\
&&1&\\
 & & &1 
\end{bmatrix} )
v\dif a\dif b\dif x.
\end{align*}
Next, we combine the terms involving $\eta\tau^{-1}$.
\begin{align*}
&q\eta\tau^{-1}\int\limits_{\OF^\times-(1+\p)}\!\int\limits_{\OF}\!\int\limits_{\OF^\times} \int\limits_{\OF^\times}\chi(ab)
\pi(
\begin{bmatrix}
1&-a\varpi^{-1}&&\\
&1&&\\
&&1&a\varpi^{-1}\\
&&&1
\end{bmatrix}\\&
\begin{bmatrix}
1&&y\varpi^{-2}&a(y+bz)\varpi^{-3}\\
&1&a^{-1}(y+bz(z-1)^{-1})\varpi^{-1}&y\varpi^{-2} \\
&&1&\\
&&&1
\end{bmatrix}
) v\dif a\dif b\dif y\dif z\\
&+q^2\chi(-1)\eta\tau^{-1}\!\int\limits_{\OF^\times-(1+\p)}\int\limits_{\OF^\times}\int\limits_{\OF^\times} \int\limits_{\OF^\times}\chi(ba(1-z)z)
\pi(\begin{bmatrix}
1&b\varpi^{-1} &&\\
&1&&\\
&&1&-b\varpi^{-1} \\
&&&1
\end{bmatrix}\\
&\begin{bmatrix}
1&&-x\varpi^{-3}& b(x-za\varpi)\varpi^{-4} \\
&1&b^{-1}(x+a\varpi)\varpi^{-2}&-x\varpi^{-3}\\
&&1&\\
&&&1
\end{bmatrix})v\dif x\dif a\dif b\dif z\\
&=q^2\chi(-1)\eta\tau^{-1}\int\limits_{\OF^\times-(1+\p)}\!\int\limits_{\OF}\!\int\limits_{\OF^\times} \int\limits_{\OF^\times}\chi(abz(1-z))
\pi(
\begin{bmatrix}
1&a\varpi^{-1}&&\\
&1&&\\
&&1&-a\varpi^{-1}\\
&&&1
\end{bmatrix}\\&
\begin{bmatrix}
1&&y\varpi^{-3}&-a(y+b(1-z)\varpi)\varpi^{-4}\\
&1&a^{-1}(-y+b\varpi)\varpi^{-2}&y\varpi^{-3} \\
&&1&\\
&&&1
\end{bmatrix}
) v\dif a\dif b\dif y\dif z.
\end{align*}
Now we combine the two terms involving $\tau^{-1}$.
\begin{align*}
&q^3\tau^{-1}\!\!\!
\int\limits_{\OF} \int\limits_{\OF^\times}\int\limits_{\OF^\times}\int\limits_{\OF^\times}\chi(ab)   \pi(
\setlength{\arraycolsep}{1pt}
\begin{bmatrix} 1 & b\varpi^{-1}  & &  \\ &1&&\\ &&1&-b\varpi^{-1} \\ &&&1 \end{bmatrix} \!\!\!
\setlength{\arraycolsep}{1pt}\begin{bmatrix} 1 &  & -(b-xa\varpi)x\varpi^{-3} & z\varpi^{-4} \\ &1&x\varpi^{-2}&-(b-xa\varpi)x\varpi^{-3}\\ &&1& \\ &&&1 \end{bmatrix}  ) v \dif a \dif b \dif x \dif z\\
&+q^2\tau^{-1} \int\limits_{\OF} \int\limits_{\OF}\int\limits_{\OF^\times}\int\limits_{\OF^\times} \chi(ab) \pi(\setlength{\arraycolsep}{1pt}\begin{bmatrix} 1 &   b \varpi^{-1} && \\  &1&& \\ &&1&-b\varpi^{-1} \\  &&&1 \end{bmatrix} 
\setlength{\arraycolsep}{1pt}\begin{bmatrix} 1 &   &(a+by)\varpi^{-2} & z\varpi^{-4} \\  &1&-y\varpi^{-1}&(a+by)\varpi^{-2} \\ &&1&\\  &&&1 \end{bmatrix}  ) v \dif a\dif b\dif y\dif z\\
&=q^3\tau^{-1}
\int\limits_{\OF} \int\limits_{\OF}\int\limits_{\OF^\times}\int\limits_{\OF^\times}\chi(ab)   \pi(
\setlength{\arraycolsep}{1pt}
\begin{bmatrix} 1 & b\varpi^{-1}  & &  \\ &1&&\\ &&1&-b\varpi^{-1} \\ &&&1 \end{bmatrix} 
\setlength{\arraycolsep}{1pt}
\begin{bmatrix} 1 &  & (xb+a\varpi)\varpi^{-3} & z\varpi^{-4} \\ &1&-x\varpi^{-2}&(xb+a\varpi)\varpi^{-3}\\ &&1& \\ &&&1 \end{bmatrix}  ) v\dif a\dif b\dif x\dif z.
\end{align*}
We now consider terms that involve the $\eta$ operator.  The sum of the third and fourth of these $\eta$ terms can be rewritten as follows:
\begin{align*}
&q\eta\int\limits_{\OF}\int\limits_{\OF^\times}\int\limits_{\OF^\times} \chi(b)
\pi(\begin{bmatrix} 1& x\varpi^{-2}& & \\ &1&&\\ &&1&-x\varpi^{-2}\\ & &&1 \end{bmatrix}\begin{bmatrix} 1& & a\varpi^{-2}&(b\varpi-ax)\varpi^{-4} \\ &1&&a\varpi^{-2} \\ &&1&\\& &&1 \end{bmatrix}  )
v\dif a\dif b\dif x\\
&+\eta\int\limits_{\OF}\!\int\limits_{\OF^\times} \int\limits_{\OF^\times}\chi(ab)
\pi(\begin{bmatrix}
1&-a\varpi^{-2}&&\\
&1&&\\
&&1&a\varpi^{-2}\\
&&&1
\end{bmatrix}\begin{bmatrix}
1&&y\varpi^{-1}&a(y+b)\varpi^{-3}\\
&1&&y\varpi^{-1} \\
&&1&\\
&&&1
\end{bmatrix}
) v\dif a\dif b\dif y\\
&=q\eta\int\limits_{\OF^\times}\int\limits_{\OF^\times}\int\limits_{\OF} \chi(b)
\pi(\begin{bmatrix} 1& x\varpi^{-2}& & \\ &1&&\\ &&1&-x\varpi^{-2}\\ & &&1 \end{bmatrix}\begin{bmatrix} 1& & a\varpi^{-2}&(b\varpi-ax)\varpi^{-4} \\ &1&&a\varpi^{-2} \\ &&1&\\& &&1 \end{bmatrix}  )
v\dif a\dif b\dif x\\
&+\eta\int\limits_{\OF}\int\limits_{\OF^\times}\int\limits_{\OF^\times} \chi(b)
\pi(\begin{bmatrix} 1& x\varpi^{-1}& & \\ &1&&\\ &&1&-x\varpi^{-1}\\ & &&1 \end{bmatrix}\begin{bmatrix} 1& & a\varpi^{-2}&(b-ax)\varpi^{-3} \\ &1&&a\varpi^{-2} \\ &&1&\\& &&1 \end{bmatrix}  )
v\dif a\dif b\dif x.\\
\end{align*}
We rewrite the first $\eta$ term after making the observation that if $z\in\OF^\times-(1+\p)$ and $f$ is a locally constant function on $\OF^\times$, then  
$$
\int\limits_{\OF^\times-A(z)}f(x) \dif x=\int\limits_{\OF^\times-(1+\p)}f((z^{-1}-1)(w^{-1}-1)^{-1})\dif w, 
$$
where $A(z)=-z^{-1}(1-z)+\p$.  Hence
\begin{align*}
&q\chi(-1)\eta\!\int\limits_{\OF^\times-(1+\p)}\int\limits_{\OF^\times}\int\limits_{\OF^\times} \int\limits_{\OF^\times - A(z)}\chi(abx)
\pi(\begin{bmatrix}
1&b \varpi^{-1} &&\\
&1&&\\
&&1&-b\varpi^{-1} \\
&&&1
\end{bmatrix}\\
&\begin{bmatrix}
1&&a\varpi^{-2} & -abx^{-1}(1+x-z)\varpi^{-3} \\
&1& -ab^{-1}xz (1-z +zx)^{-1}\varpi^{-1} & a\varpi^{-2} \\
&&1&\\
&&&1
\end{bmatrix}
) v\dif x\dif a\dif b\dif z\\
&=q\chi(-1)\eta\!\int\limits_{\OF^\times-(1+\p)}\int\limits_{\OF^\times}\int\limits_{\OF^\times} \int\limits_{\OF^\times - (1+\p)}\chi(abz(1-z)w(1-w))
\pi(\begin{bmatrix}
1&b \varpi^{-1} &&\\
&1&&\\
&&1&-b\varpi^{-1} \\
&&&1
\end{bmatrix}\\
&\begin{bmatrix}
1&&a\varpi^{-2} & -ab(1+zw^{-1}(1-w))\varpi^{-3} \\
&1& -ab^{-1}w\varpi^{-1} & a\varpi^{-2} \\
&&1&\\
&&&1
\end{bmatrix}
) v\dif w\dif a\dif b\dif z\\
&=q\chi(-1)\eta\!\int\limits_{\OF^\times-(1+\p)}\int\limits_{\OF^\times}\int\limits_{\OF^\times} \int\limits_{\OF^\times - (1+\p)}\chi(az^{-1}(1-z)bw^{-1}(1-w))
\pi(\begin{bmatrix}
1&b \varpi^{-1} &&\\
&1&&\\
&&1&-b\varpi^{-1} \\
&&&1
\end{bmatrix}\\
&\begin{bmatrix}
1&&a\varpi^{-2} & -ab(1+zw^{-1}(1-w))\varpi^{-3} \\
&1& -ab^{-1}w\varpi^{-1} & a\varpi^{-2} \\
&&1&\\
&&&1
\end{bmatrix}
) v\dif w\dif a\dif b\dif z\\
&=q\chi(-1)\eta\!\int\limits_{\OF^\times-(1+\p)}\int\limits_{\OF^\times}\int\limits_{\OF^\times} \int\limits_{\OF^\times - (1+\p)}\chi(a(z^{-1}-1)b(w^{-1}-1))
\pi(\begin{bmatrix}
1&b \varpi^{-1} &&\\
&1&&\\
&&1&-b\varpi^{-1} \\
&&&1
\end{bmatrix}\\
&\begin{bmatrix}
1&&a\varpi^{-2} & -ab(1+z(w^{-1}-1))\varpi^{-3} \\
&1& -ab^{-1}w\varpi^{-1} & a\varpi^{-2} \\
&&1&\\
&&&1
\end{bmatrix}
) v\dif w\dif a\dif b\dif z\\
&=q\chi(-1)\eta\!\int\limits_{\OF^\times-(1+\p)}\int\limits_{\OF^\times}\int\limits_{\OF^\times} \int\limits_{\OF^\times - (1+\p)}\chi(a(z-1)b(w-1))
\pi(\begin{bmatrix}
1&b \varpi^{-1} &&\\
&1&&\\
&&1&-b\varpi^{-1} \\
&&&1
\end{bmatrix}\\
&\begin{bmatrix}
1&&a\varpi^{-2} & -ab(1+z^{-1}(w-1))\varpi^{-3} \\
&1& -ab^{-1}w^{-1}\varpi^{-1} & a\varpi^{-2} \\
&&1&\\
&&&1
\end{bmatrix}
) v\dif w\dif a\dif b\dif z\\
&=q\chi(-1)\eta\!\int\limits_{\OF^\times-(-1+\p)}\int\limits_{\OF^\times}\int\limits_{\OF^\times} \int\limits_{\OF^\times - (-1+\p)}\chi(aybx)
\pi(\begin{bmatrix}
1&b \varpi^{-1} &&\\
&1&&\\
&&1&-b\varpi^{-1} \\
&&&1
\end{bmatrix}\\
&\begin{bmatrix}
1&&a\varpi^{-2} & -ab(1+(1+y)^{-1}x)\varpi^{-3} \\
&1& -ab^{-1}(1+x)^{-1}\varpi^{-1} & a\varpi^{-2} \\
&&1&\\
&&&1
\end{bmatrix}
) v\dif x\dif a\dif b\dif y\\
&=q\eta\!\int\limits_{\OF^\times-(1+\p)}\int\limits_{\OF^\times}\int\limits_{\OF^\times} \int\limits_{\OF^\times - (1+\p)}\chi(abxy)
\pi(\begin{bmatrix}
1&b \varpi^{-1} &&\\
&1&&\\
&&1&-b\varpi^{-1} \\
&&&1
\end{bmatrix}\\
&\begin{bmatrix}
1&&-a\varpi^{-2} & ab(1-(1-y)^{-1}x)\varpi^{-3} \\
&1& ab^{-1}(1-x)^{-1}\varpi^{-1} & -a\varpi^{-2} \\
&&1&\\
&&&1
\end{bmatrix}
) v\dif x\dif a\dif b\dif y.
\end{align*}
Finally, we are able to eliminate the factor $\chi(-1)$ from all terms using an appropriate change of variables.  Substituting the simplified terms into the formula for $T_\chi(v)$, we obtain the result.
\end{proof}
The remainder of this section is devoted to calculating the formulas for \eqref{termoneeq}, \eqref{termtwoeq}, \eqref{termthreeeq}, and \eqref{termfoureq}.  Except for \eqref{termthreeeq}, the terms are each calculated in a single lemma.  The lemmas for calculating \eqref{termthreeeq} appear after those of the more straightforward terms, and there is some explanation of the strategy used to complete this calculation.
\begin{lemma}\label{termonelemma}
If $v\in V(0)$, then we have that \eqref{termoneeq} is given by
\begin{align*}
&q\tau\int\limits_{\OF} \int\limits_{\OF^\times}\int\limits_{\OF^\times}\chi(ab)  
\pi( \begin{bmatrix} 1 & -a\varpi^{-2} & & \\ &1&& \\ &&1&a\varpi^{-2} \\ &&&1 \end{bmatrix} \begin{bmatrix} 1 &  & b \varpi^{-1} & z\varpi^{-4}\\ &1&&b\varpi^{-1} \\ &&1&  \\ &&&1 \end{bmatrix}  ) v\dif a\dif b\dif z\\
&+q^3\tau^{-1}\!\!\!
\int\limits_{\OF}\! \int\limits_{\OF^\times}\!\int\limits_{\OF^\times}\!\int\limits_{\OF^\times}\chi(ab)   \pi(
\setlength{\arraycolsep}{1pt}\begin{bmatrix} 1 & b\varpi^{-1}  & &  \\ &1&&\\ &&1&-b\varpi^{-1} \\ &&&1 \end{bmatrix} \!\!\!
\setlength{\arraycolsep}{1pt}\begin{bmatrix} 1 &  & -(b-xa\varpi)x\varpi^{-3} & z\varpi^{-4} \\ &1&x\varpi^{-2}&-(b-xa\varpi)x\varpi^{-3}\\ &&1& \\ &&&1 \end{bmatrix}  ) v\dif a\dif b\dif x\dif z\\
&+q^2
\int\limits_{\OF} \int\limits_{\OF^\times}\int\limits_{\OF^\times}\int\limits_{\OF^\times}\chi(ab)   \pi(
\setlength{\arraycolsep}{1pt}\begin{bmatrix} 1 & -(a+xb)\varpi^{-1}  & & \\ &1&& \\ &&1&(a+xb) \varpi^{-1} \\ &&&1 \end{bmatrix} 
\setlength{\arraycolsep}{1pt}\begin{bmatrix} 1 &   & b\varpi^{-2} & z \varpi^{-4} \\ &1&x^{-1}\varpi^{-1}&b\varpi^{-2} \\ &&1& \\ &&&1 \end{bmatrix} ) v\dif a\dif b\dif x\dif z.
\end{align*}
\end{lemma}
The \hyperlink{termonelemmaproof}{proof} is given below in Section \ref{lemmaproofsec}.

\begin{lemma}\label{termtwolemma}
If $v\in V(0)$, then we have that \eqref{termtwoeq} is given by
\begin{align*}
&q^2\tau^{-1} \int\limits_{\OF} \int\limits_{\OF}\int\limits_{\OF^\times}\int\limits_{\OF^\times} \chi(ab) \pi(
\setlength{\arraycolsep}{1pt}\begin{bmatrix} 1 &   b \varpi^{-1} && \\  &1&& \\ &&1&-b\varpi^{-1} \\  &&&1 \end{bmatrix} 
\setlength{\arraycolsep}{1pt}\begin{bmatrix} 1 &   &(a+by)\varpi^{-2} & z\varpi^{-4} \\  &1&-y\varpi^{-1}&(a+by)\varpi^{-2} \\ &&1&\\  &&&1 \end{bmatrix}  ) v \dif a\dif b\dif y\dif z.
\end{align*}
\end{lemma}
The straightforward \hyperlink{termtwolemmaproof}{proof} is given below in Section \ref{lemmaproofsec}.

\begin{lemma}\label{termfourlemma}
If $v\in V(0)$, then we have that \eqref{termfoureq} is given by
\begin{align*}
&\int\limits_{\OF}\!\int\limits_{\OF^\times} \int\limits_{\OF^\times}\chi(ab)
\eta\pi(\begin{bmatrix}
1&-a\varpi^{-2}&&\\
&1&&\\
&&1&a\varpi^{-2}\\
&&&1
\end{bmatrix}\begin{bmatrix}
1&&y\varpi^{-1}&a(y+b)\varpi^{-3}\\
&1&&y\varpi^{-1} \\
&&1&\\
&&&1
\end{bmatrix}
) v\dif a\dif b\dif y\\
&+q\int\limits_{\OF^\times-(1+\p)}\!\int\limits_{\OF}\!\int\limits_{\OF^\times} \int\limits_{\OF^\times}\chi(ab)
\eta\tau^{-1}\pi(
\begin{bmatrix}
1&-a\varpi^{-1}&&\\
&1&&\\
&&1&a\varpi^{-1}\\
&&&1
\end{bmatrix}\\&
\begin{bmatrix}
1&&y\varpi^{-2}&a(y+bz)\varpi^{-3}\\
&1&a^{-1}(y+bz(z-1)^{-1})\varpi^{-1}&y\varpi^{-2} \\
&&1&\\
&&&1
\end{bmatrix}
) v\dif a\dif b\dif y\dif z\\
&+\!\int\limits_{\OF}\!\int\limits_{\OF^\times} \int\limits_{\OF^\times}\chi(a)
\eta^2\tau^{-2}\pi(\begin{bmatrix}
1&&b\varpi^{-2}&-a\varpi^{-1}\\
&1&y\varpi^{-3}&b\varpi^{-2} \\
&&1&\\
&&&1
\end{bmatrix}
) v\dif a\dif b\dif y.
\end{align*}
\end{lemma}
The \hyperlink{termfourlemmaproof}{proof} is given below in Section \ref{lemmaproofsec}.

Finally, we consider the most difficult summand in the expression for $T_\chi(v)$.
\begin{lemma}\label{termthreelemma}
If $v\in V(0)$, then we have that \eqref{termthreeeq} is given by
\begin{align*}
&q^2\chi(-1)\!\int\limits_{\OF^\times-(1+\p)}\int\limits_{\OF^\times}\int\limits_{\OF^\times} \int\limits_{\OF^\times}\chi(abx(1-z))
\eta\tau^{-1}\pi(\begin{bmatrix}
1&b\varpi^{-1} &&\\
&1&&\\
&&1&-b\varpi^{-1} \\
&&&1
\end{bmatrix}\\
&
\begin{bmatrix}
1&&a \varpi^{-3}& -ab(1+x\varpi)\varpi^{-4} \\
&1&-ab^{-1}(1+xz^{-1}\varpi)^{-1}\varpi^{-2}&a \varpi^{-3} \\
&&1&\\
&&&1
\end{bmatrix}
)v\dif x\dif a\dif b\dif z\\
&+\!\int\limits_{\OF^\times-(1+\p)}\int\limits_{\OF^\times}\int\limits_{\OF^\times} \int\limits_{\OF}\chi(abz(1-z))
\eta\tau\pi(
\setlength{\arraycolsep}{1pt}\begin{bmatrix}
1&b\varpi^{-2}&\\
&1&&\\
 &&1&-b\varpi^{-2}\\
& &&1
\end{bmatrix}
\setlength{\arraycolsep}{1pt}\begin{bmatrix}
1 &  & a\varpi^{-1}&-ab(1-z)\varpi^{-3} \\
&1& & a\varpi^{-1}\\
&&1 & \\
&&&1
\end{bmatrix}
) v\dif x\dif a\dif b\dif z\\
&+q\chi(-1)\!\int\limits_{\OF^\times-(1+\p)}\int\limits_{\OF^\times}\int\limits_{\OF^\times} \int\limits_{\OF^\times - A(z)}\chi(abx)
\eta\pi(\begin{bmatrix}
1&b \varpi^{-1} &&\\
&1&&\\
&&1&-b\varpi^{-1} \\
&&&1
\end{bmatrix}\\
&\begin{bmatrix}
1&&a\varpi^{-2} & abx^{-1}(1+x-z)\varpi^{-3} \\
&1& -ab^{-1}xz (1-z +zx)^{-1}\varpi^{-1} & a\varpi^{-2} \\
&&1&\\
&&&1
\end{bmatrix}
) v\dif x\dif a\dif b\dif z\\
&+\chi(-1)\!\int\limits_{\OF^\times-(1+\p)}\int\limits_{\OF^\times}\int\limits_{\OF^\times} \int\limits_{\OF} \chi(b(1-z))
\eta\pi(\begin{bmatrix}
1&&a\varpi^{-2}&-b\varpi^{-3}\\
&1&-a^2b^{-1}z\varpi^{-1}&a\varpi^{-2}\\
&&1&\\
&&&1
\end{bmatrix})v\dif x\dif a\dif b\dif z\\
&+q\chi(-1)\int\limits_{\OF^\times}\int\limits_{\OF^\times}\int\limits_{\OF^\times} \chi(b)
\eta^2\tau^{-2}\pi(
\begin{bmatrix}
1&&a\varpi^{-2}&b\varpi^{-1}\\
&1&x \varpi^{-4}&a\varpi^{-2}\\
&&1&\\
 & & &1 
\end{bmatrix} )
v\dif a\dif b\dif x\\
&+q^{-1}\chi(-1)\int\limits_{\OF^\times}\int\limits_{\OF^\times}\int\limits_{\OF^\times} \chi(bx)
\eta^2\pi(
\begin{bmatrix}
1&&a\varpi^{-2}&  b(1+x\varpi)\varpi^{-2}\\
&1&a^2b^{-1}\varpi^{-2}&a\varpi^{-2}\\
&&1&\\
&&&1
\end{bmatrix} )
v\dif a\dif b\dif x\\
&+q^{-2}\chi(-1)\int\limits_{\OF^\times}\int\limits_{\OF^\times}\int\limits_{\OF^\times} \chi(b)
\eta^2\tau\pi(
\begin{bmatrix}
1&a\varpi^{-1}& &\\
&1&&\\
 &&1&-a\varpi^{-1}\\
& &&1
\end{bmatrix}
\begin{bmatrix}
1&&&b\varpi^{-1} \\
 & 1& x\varpi^{-1}& \\
&&1& \\
&& & 1 
\end{bmatrix})
v\dif a\dif b\dif x\\
&+q^{-3}\chi(-1)\int\limits_{\OF^\times}\int\limits_{\OF^\times} \chi(b)
\eta^2\tau^2\pi(\begin{bmatrix}
1&a\varpi^{-2}& &b\varpi^{-1}\\
&1&&\\
 &&1&-a\varpi^{-2}\\
& &&1
\end{bmatrix} )
v\dif a\dif b\\
&+q\int\limits_{\OF}\int\limits_{\OF^\times}\int\limits_{\OF^\times} \chi(b)
\eta\pi(\begin{bmatrix} 1& x\varpi^{-2}& & \\ &1&&\\ &&1&-x\varpi^{-2}\\ & &&1 \end{bmatrix}\begin{bmatrix} 1& & a\varpi^{-2}&(b\varpi-ax)\varpi^{-4} \\ &1&&a\varpi^{-2} \\ &&1&\\& &&1 \end{bmatrix}  )
v\dif a\dif b\dif x.
\end{align*}
\end{lemma}
\begin{proof} 
The proof follows by combining the formulas in Lemmas \ref{termthreepreplemma}, \ref{11lemma}, \ref{12lemma}, \ref{13lemma}, \ref{100eqlemma}, and \ref{101eqlemma}, which are given below.
\end{proof}
The calculation of the \eqref{termthreeeq} term is the most delicate.  We begin this calculation with a preparatory lemma  which breaks the term into four pieces.  The majority of our calculations will be devoted to handling the first of these terms.
\begin{lemma}\label{termthreepreplemma}
If $v\in V(0)$, then we have that \eqref{termthreeeq} is given by
\begin{align}
&q^2\!\int\limits_{\OF^\times}\int\limits_{\OF}\int\limits_{\OF^\times}\int\limits_{\OF^\times} \chi(b)
\eta\pi(\begin{bmatrix} 1& & & \\ &1&& \\ &&1& \\z\varpi^{3} &&&1 \end{bmatrix}\begin{bmatrix} 1& x\varpi^{-2}& a\varpi^{-2}&b\varpi^{-3} \\ &1&&a\varpi^{-2} \\ &&1&-x\varpi^{-2}\\ & &&1 \end{bmatrix} )
v\dif a\dif b\dif x\dif z\label{10eq}\\
&+q\int\limits_{\OF^\times}\int\limits_{\OF}\int\limits_{\OF^\times}\int\limits_{\OF^\times} \chi(b)
\eta\pi(\begin{bmatrix} 1& & & \\ &1&& \\ &&1& \\z\varpi^{4} &&&1 \end{bmatrix}\begin{bmatrix} 1& x\varpi^{-2}& a\varpi^{-2}&b\varpi^{-3} \\ &1&&a\varpi^{-2} \\ &&1&-x\varpi^{-2}\\ & &&1 \end{bmatrix} )
v\dif a\dif b\dif x\dif z\label{11eq}\\
&+\int\limits_{\OF^\times}\int\limits_{\OF}\int\limits_{\OF^\times}\int\limits_{\OF^\times} \chi(b)
\eta\pi(\begin{bmatrix} 1& & & \\ &1&& \\ &&1& \\z\varpi^{5} &&&1 \end{bmatrix}\begin{bmatrix} 1& x\varpi^{-2}& a\varpi^{-2}&b\varpi^{-3} \\ &1&&a\varpi^{-2} \\ &&1&-x\varpi^{-2}\\ & &&1 \end{bmatrix} )
v\dif a\dif b\dif x\dif z\label{12eq}\\
&+q^{-1}\!\int\limits_{\OF}\int\limits_{\OF}\int\limits_{\OF^\times}\int\limits_{\OF^\times} \chi(b)
\eta\pi(\begin{bmatrix} 1& & & \\ &1&& \\ &&1& \\z\varpi^{6} &&&1 \end{bmatrix}\begin{bmatrix} 1& x\varpi^{-2}& a\varpi^{-2}&b\varpi^{-3} \\ &1&&a\varpi^{-2} \\ &&1&-x\varpi^{-2}\\ & &&1 \end{bmatrix} )
v\dif a\dif b\dif x\dif z\label{13eq}.
\end{align}
\end{lemma}
The \hyperlink{termthreeepreplemmaproof}{proof} of this lemma is given below in Section \ref{lemmaproofsec}.
The following three lemmas calculate the straightforward terms from the decomposition in Lemma \ref{termthreepreplemma}.  The strategy in each case is to conjugate by the integral lower triangular matrix, use the invariance of $v$ under $\GSp(4,\OF)$, and apply appropriate changes of variables.
\begin{lemma}\label{11lemma}
If $v\in V(0)$, then we have that \eqref{11eq} is given by
\begin{align*}
&(q-1)\int\limits_{\OF}\int\limits_{\OF^\times}\int\limits_{\OF^\times} \chi(b)
\eta\pi(\begin{bmatrix} 1& x\varpi^{-2}& a\varpi^{-2}&b\varpi^{-3} \\ &1&&a\varpi^{-2} \\ &&1&-x\varpi^{-2}\\ & &&1 \end{bmatrix})v\dif a\dif b\dif x.
\end{align*}
\end{lemma}
The \hyperlink{11lemmaproof}{proof} of this lemma is given below in Section \ref{lemmaproofsec}.
\begin{lemma}
\label{12lemma}
If $v \in V(0)$, then we have that \eqref{12eq} is given by
\begin{align*}
&(1-q^{-1})\int\limits_{\OF}\int\limits_{\OF^\times}\int\limits_{\OF^\times} \chi(b)
\eta\pi(\begin{bmatrix} 1& x\varpi^{-2}& a\varpi^{-2}&b\varpi^{-3} \\ &1&&a\varpi^{-2} \\ &&1&-x\varpi^{-2}\\ & &&1 \end{bmatrix} )
v\dif a\dif b\dif x.
\end{align*}
\end{lemma}
The \hyperlink{12lemmaproof}{proof} of this lemma is given below in Section \ref{lemmaproofsec}.
\begin{lemma}
If $v\in V(0)$, then we have that \eqref{13eq} is given by
\begin{align*}
&q^{-1}\int\limits_{\OF}\int\limits_{\OF^\times}\int\limits_{\OF^\times} \chi(b)
\eta\pi(\begin{bmatrix} 1& x\varpi^{-2}& a\varpi^{-2}&b\varpi^{-3} \\ &1&&a\varpi^{-2} \\ &&1&-x\varpi^{-2}\\ & &&1 \end{bmatrix})
v\dif a\dif b\dif x.
\end{align*}
\label{13lemma}
\end{lemma}
The \hyperlink{13lemmaproof}{proof} of this lemma is given below in Section \ref{lemmaproofsec}.
Finally, to calculate the remaining term \eqref{10eq} of \eqref{termthreeeq} in Lemma \ref{termthreepreplemma}, let 
$$
D(a,b)=\begin{bmatrix}1&&&\\&ab^{-1}&&\\&&a^{-1}&\\&&&b^{-1}\end{bmatrix}, 
$$
for $a,b\in F^\times$.  Then, we decompose \eqref{10eq} into two further terms as follows:
\begin{align}
&\nonumber q^2\!\int\limits_{\OF^\times}\int\limits_{\OF}\int\limits_{\OF^\times}\int\limits_{\OF^\times} \chi(b)
\eta\pi(\begin{bmatrix} 1& & & \\ &1&& \\ &&1& \\z\varpi^{3} &&&1 \end{bmatrix}\begin{bmatrix} 1& x\varpi^{-2}& a\varpi^{-2}&b\varpi^{-3} \\ &1&&a\varpi^{-2} \\ &&1&-x\varpi^{-2}\\ & &&1 \end{bmatrix} )
v\dif a\dif b\dif x\dif z\\
&=\nonumber q^2\!\int\limits_{\OF^\times}\int\limits_{\OF}\int\limits_{\OF^\times}\int\limits_{\OF^\times} \chi(-b)
\eta\pi(\begin{bmatrix} 1& & & \\ &1&& \\ &&1& \\z\varpi^{3} &&&1 \end{bmatrix}\begin{bmatrix} 1& x\varpi^{-2}& a\varpi^{-2}&-b\varpi^{-3} \\ &1&&a\varpi^{-2} \\ &&1&-x\varpi^{-2}\\ & &&1 \end{bmatrix} )
v\dif a\dif b\dif x\dif z\\
&=\nonumber q^2\chi(-1)\!\int\limits_{\OF^\times}\int\limits_{\OF}\int\limits_{\OF^\times}\int\limits_{\OF^\times} \chi(b)
\eta\pi(D(a,b)\setlength{\arraycolsep}{1pt}\begin{bmatrix} 1& & & \\ &1&& \\ &&1& \\bz\varpi^{3} &&&1 \end{bmatrix}
\setlength{\arraycolsep}{1pt}\begin{bmatrix} 1& ab^{-1}x\varpi^{-2}& \varpi^{-2}&-\varpi^{-3} \\ &1&&\varpi^{-2} \\ &&1&-ab^{-1}x\varpi^{-2}\\ & &&1 \end{bmatrix} )
v\dif a\dif b\dif x\dif z\\
&=\nonumber \chi(-1)q^2\!\int\limits_{\OF^\times}\int\limits_{\OF}\int\limits_{\OF^\times}\int\limits_{\OF^\times} \chi(b)
\eta\pi(D(a,b)\begin{bmatrix} 1& & & \\ &1&& \\ &&1& \\z\varpi^{3} &&&1 \end{bmatrix}
\begin{bmatrix} 1& x\varpi^{-2}& \varpi^{-2}&-\varpi^{-3} \\ &1&&\varpi^{-2} \\ &&1&-x\varpi^{-2}\\ & &&1 \end{bmatrix} )
v\dif a\dif b\dif x\dif z\\
&=\label{100eq}q^2\chi(-1)\!\int\limits_{\OF^\times-(1+\p)}\int\limits_{\OF}\int\limits_{\OF^\times}\int\limits_{\OF^\times} \chi(b)
\eta\pi(D(a,b)\setlength{\arraycolsep}{1pt}\begin{bmatrix} 1& & & \\ &1&& \\ &&1& \\z\varpi^{3} &&&1 \end{bmatrix}
\setlength{\arraycolsep}{1pt}\begin{bmatrix} 1& x\varpi^{-2}& \varpi^{-2}&-\varpi^{-3} \\ &1&&\varpi^{-2} \\ &&1&-x\varpi^{-2}\\ & &&1 \end{bmatrix} )
v\dif a\dif b\dif x\dif z\\
&+\label{101eq}q^2\chi(-1)\!\int\limits_{1+\p}\int\limits_{\OF}\int\limits_{\OF^\times}\int\limits_{\OF^\times} \chi(b)
\eta\pi(D(a,b)\setlength{\arraycolsep}{1pt}\begin{bmatrix} 1& & & \\ &1&& \\ &&1& \\z\varpi^{3} &&&1 \end{bmatrix}
\setlength{\arraycolsep}{1pt}\begin{bmatrix} 1& x\varpi^{-2}& \varpi^{-2}&-\varpi^{-3} \\ &1&&\varpi^{-2} \\ &&1&-x\varpi^{-2}\\ & &&1 \end{bmatrix} )
v\dif a\dif b\dif x\dif z.
\end{align}

\begin{lemma}\label{100eqlemma}
If $v\in V(0)$, then we have that \eqref{100eq} is given by
\begin{align*}
&q^2\chi(-1)\!\int\limits_{\OF^\times-(1+\p)}\int\limits_{\OF^\times}\int\limits_{\OF^\times} \int\limits_{\OF^\times}\chi(ba(1-z)z)
\eta\tau^{-1}\pi(\begin{bmatrix}
1&b\varpi^{-1} &&\\
&1&&\\
&&1&-b\varpi^{-1} \\
&&&1
\end{bmatrix}\\
&\begin{bmatrix}
1&&-x\varpi^{-3}& b(x-za\varpi)\varpi^{-4} \\
&1&b^{-1}(x+a\varpi)\varpi^{-2}&-x\varpi^{-3}\\
&&1&\\
&&&1
\end{bmatrix})v\dif x\dif a\dif b\dif z\\
&+\!\int\limits_{\OF^\times-(1+\p)}\int\limits_{\OF^\times}\int\limits_{\OF^\times} \int\limits_{\OF}\chi(abz(1-z))
\eta\tau\pi(\setlength{\arraycolsep}{1pt}\begin{bmatrix}
1&b\varpi^{-2}&\\
&1&&\\
 &&1&-b\varpi^{-2}\\
& &&1
\end{bmatrix}
\begin{bmatrix}
1 &  & a\varpi^{-1}&-ab(1-z)\varpi^{-3} \\
&1& & a\varpi^{-1}\\
&&1 & \\
&&&1
\end{bmatrix}
) v\dif x\dif a\dif b\dif z\\
&+q\chi(-1)\!\int\limits_{\OF^\times-(1+\p)}\int\limits_{\OF^\times}\int\limits_{\OF^\times} \int\limits_{\OF^\times - A(z)}\chi(abx)
\eta\pi(\begin{bmatrix}
1&b \varpi^{-1} &&\\
&1&&\\
&&1&-b\varpi^{-1} \\
&&&1
\end{bmatrix}\\
&\begin{bmatrix}
1&&a\varpi^{-2} & -abx^{-1}(1+x-z)\varpi^{-3} \\
&1& -ab^{-1}xz (1-z +zx)^{-1}\varpi^{-1} & a\varpi^{-2} \\
&&1&\\
&&&1
\end{bmatrix}
) v\dif x\dif a\dif b\dif z\\
&+\chi(-1)\!\int\limits_{\OF^\times-(1+\p)}\int\limits_{\OF^\times}\int\limits_{\OF^\times}\chi(b(1-z))
\eta\pi(\begin{bmatrix}
1&&a\varpi^{-2}&-b\varpi^{-3}\\
&1&-a^2b^{-1}z\varpi^{-1}&a\varpi^{-2}\\
&&1&\\
&&&1
\end{bmatrix})v\dif a\dif b\dif z.
\end{align*}
\end{lemma}
The \hyperlink{100eqlemmaproof}{proof} of this lemma is both lengthy and delicate and is given below in Section \ref{lemmaproofsec}.
The identities from the next lemma will be used in the calculation of the remaining term \eqref{101eq}.
\begin{lemma}
\label{101preplemma}
Let  $v\in V(0)$ and $z \in 1+\p$.  
\begin{enumerate}
 \item Assume that $x \in \OF^\times$. Then:                                                                  
\begin{align*}
&\pi( \begin{bmatrix}
1&& &\\
&1&&\\
 &&1&\\
 z\varpi^{3}&&&1
\end{bmatrix}
\begin{bmatrix}
1&x\varpi^{-2}&\varpi^{-2} &-\varpi^{-3}\\
&1&&\varpi^{-2}\\
 &&1&-x\varpi^{-2}\\
& &&1
\end{bmatrix}) v\\
&= \pi( \begin{bmatrix}
\varpi^{-1}&&&\\
&\varpi^{2}&&\\
&&\varpi^{-2}&\\
 & & &\varpi
\end{bmatrix}
\begin{bmatrix}
1&&(xz)^{-1}\varpi^{-2}&z^{-1} \varpi^{-1}\\
&1&-z(xz+(1-z) \varpi)^{-1} \varpi^{-4}&(xz)^{-1}\varpi^{-2}\\
&&1&\\
 & & &1 
\end{bmatrix}
) v.
\end{align*}
\item Assume that $x \in \OF^\times$. Then: 
\begin{align*}
&\pi( \begin{bmatrix}
1&& &\\
&1&&\\
 &&1&\\
 z\varpi^{3}&&&1
\end{bmatrix}
\begin{bmatrix}
1&x\varpi^{-1}&\varpi^{-2} &-\varpi^{-3}\\
&1&&\varpi^{-2}\\
 &&1&-x\varpi^{-1}\\
& &&1
\end{bmatrix}) v\\
&=\pi(\setlength{\arraycolsep}{1pt} \begin{bmatrix}
\varpi^{-1}&&&\\
&\varpi&&\\
&&\varpi^{-1}&\\
 & & &\varpi
\end{bmatrix}
\begin{bmatrix}
1&&(1-z+zx)^{-1}\varpi^{-2}&-z^{-2}(1-z+x)^{-1}\varpi^{-1}+z^{-1}\varpi^{-1}\\
&1&-(1-z+zx)^{-1} z \varpi^{-3}&(1-z+zx)^{-1}\varpi^{-2}\\
&&1&\\
 & & &1 
\end{bmatrix}
) v.
\end{align*}
\item Assume that $x \in (1-z^{-1})\varpi^{-1}+\OF^\times$. Then:
\begin{align*}
&\pi( \begin{bmatrix}
1&& &\\
&1&&\\
 &&1&\\
 z\varpi^{3}&&&1
\end{bmatrix}
\begin{bmatrix}
1&x&\varpi^{-2} &-\varpi^{-3}\\
&1&&\varpi^{-2}\\
 &&1&-x\\
& &&1
\end{bmatrix}) v\\
&= \pi( 
 \begin{bmatrix}
\varpi^{-1}&& & \\
&1&&\\
 &&1&\\
&&&\varpi
\end{bmatrix}
\begin{bmatrix}
1&&-z^{-1}w^{-1}\varpi^{-2}&  (1+zw\varpi) z^{-2}w^{-1}\varpi^{-2}\\
&1&w^{-1}\varpi^{-2}&-z^{-1}w^{-1}\varpi^{-2}\\
&&1&\\
&&&1
\end{bmatrix}
) v
\end{align*}
where $w=(1-x\varpi-z^{-1})\varpi^{-1} \in \OF^\times$. 
\item Assume that $x \in (1-z^{-1})\varpi^{-1}+\varpi \OF^\times$. Then:
\begin{align*}
&\pi( \begin{bmatrix}
1&& &\\
&1&&\\
 &&1&\\
 z\varpi^{3}&&&1
\end{bmatrix}
\begin{bmatrix}
1&x&\varpi^{-2} &-\varpi^{-3}\\
&1&&\varpi^{-2}\\
 &&1&-x\\
& &&1
\end{bmatrix}) v\\
&= \pi( 
\setlength{\arraycolsep}{3pt} \begin{bmatrix}
\varpi^{-1}&& & \\
&\varpi^{-1}&&\\
 &&\varpi&\\
&&&\varpi
\end{bmatrix}\setlength{\arraycolsep}{2pt}
\begin{bmatrix}
1&-z^{-1}\varpi^{-1}& &\\
&1&&\\
 &&1&z^{-1}\varpi^{-1}\\
& &&1
\end{bmatrix}
\begin{bmatrix}
1&&&z^{-1}\varpi^{-1} \\
 & 1& w^{-1}\varpi^{-1} & \\
&&1& \\
&& & 1 
\end{bmatrix}) v
\end{align*}
where $w=(1-x\varpi -z^{-1}) \varpi^{-2} \in \OF^\times$. 
\item Assume that $x \in (1-z^{-1}) \varpi^{-1} +\p^2$. Then:
\begin{align*}
\pi( \setlength{\arraycolsep}{3pt}\begin{bmatrix}
1&& &\\
&1&&\\
 &&1&\\
 z\varpi^{3}&&&1
\end{bmatrix}\!\!\!
\begin{bmatrix}
1&x&\varpi^{-2} &-\varpi^{-3}\\
&1&&\varpi^{-2}\\
 &&1&-x\\
& &&1
\end{bmatrix}) v
=\pi(  \begin{bmatrix}
\varpi^{-1}\!\!\!\!\!\!&& & \\
&\varpi^{-2}\!\!\!&&\\
 &&\varpi^2\!\!\!&\\
&&&\varpi
\end{bmatrix}\!\!\!
\begin{bmatrix}
1&-z^{-1}\varpi^{-2}& &\varpi^{-1}\\
&1&&\\
 &&1&z^{-1}\varpi^{-2}\\
& &&1
\end{bmatrix}
) v.
\end{align*}
\end{enumerate}
\end{lemma}
The \hyperlink{101preplemmaproof}{proof} of this lemma is given below in Section \ref{lemmaproofsec}.
\begin{lemma}\label{101eqlemma}
If $v\in V(0)$, then we have that \eqref{101eq} is given by
\begin{align*}
&q\chi(-1)\int\limits_{\OF^\times}\int\limits_{\OF^\times}\int\limits_{\OF^\times} \chi(b)
\eta^2\tau^{-2}\pi(
\begin{bmatrix}
1&&a\varpi^{-2}&b\varpi^{-1}\\
&1&x \varpi^{-4}&a\varpi^{-2}\\
&&1&\\
 & & &1 
\end{bmatrix} )
v\dif a\dif b\dif x\\
&+\chi(-1)\int\limits_{\OF^\times}\int\limits_{\OF^\times}\int\limits_{\OF^\times} \chi(bx)
\eta^2\tau^{-1}\pi(\begin{bmatrix}
1&&a\varpi^{-2}&b(x-1)\varpi^{-1}\\
&1&-a^2b^{-1}\varpi^{-3}&a\varpi^{-2}\\
&&1&\\
 & & &1 
\end{bmatrix} )
v\dif a\dif b\dif x\\
&+q^{-1}\chi(-1)\int\limits_{\OF^\times}\int\limits_{\OF^\times}\int\limits_{\OF^\times} \chi(bx)
\eta^2\pi(
\begin{bmatrix}
1&&a\varpi^{-2}&  b(1+x\varpi)\varpi^{-2}\\
&1&a^2b^{-1}\varpi^{-2}&a\varpi^{-2}\\
&&1&\\
&&&1
\end{bmatrix} )
v\dif a\dif b\dif x\\
&+q^{-2}\chi(-1)\int\limits_{\OF^\times}\int\limits_{\OF^\times}\int\limits_{\OF^\times} \chi(b)
\eta^2\tau\pi(
\begin{bmatrix}
1&a\varpi^{-1}& &\\
&1&&\\
 &&1&-a\varpi^{-1}\\
& &&1
\end{bmatrix}
\begin{bmatrix}
1&&&b\varpi^{-1} \\
 & 1& x\varpi^{-1}& \\
&&1& \\
&& & 1 
\end{bmatrix})
v\dif a\dif b\dif x\\
&+q^{-3}\chi(-1)\int\limits_{\OF^\times}\int\limits_{\OF^\times} \chi(b)
\eta^2\tau^2\pi(\begin{bmatrix}
1&a\varpi^{-2}& &b\varpi^{-1}\\
&1&&\\
 &&1&-a\varpi^{-2}\\
& &&1
\end{bmatrix} )
v\dif a\dif b.
\end{align*}
\end{lemma}
The \hyperlink{101eqlemmaproof}{proof} of this lemma is given below in Section \ref{lemmaproofsec}.
\section{Proofs of the lemmas in the local calculation}\label{lemmaproofsec}
\begin{proof}[\hypertarget{termonelemmaproof}{Proof of} \ref{termonelemma}]
In this proof we use the methods discussed at the beginning of Section \ref{localsec} to obtain an upper triangular form of summand \eqref{termoneeq} of the operator $T_\chi$.
\begin{align*}
&q^3\int\limits_{\OF}\int\limits_{\OF}\int\limits_{\OF^\times}\int\limits_{\OF^\times}
\chi(ab)\pi(
 ) v\dif a\dif b\dif x\dif z.
\end{align*}
This completes the calculation of the first term \eqref{termoneeq}.
\end{proof}
\begin{proof}[\hypertarget{termtwolemmaproof}{Proof of} \ref{termtwolemma}]
This calculation is straightforward and uses only the invariance of $v$ under $\GSp(4,\OF)$.  We have
\begin{align*}
& q^3\int\limits_{\OF}\int\limits_{\p}\int\limits_{\OF^\times}\int\limits_{\OF^\times}
\chi(ab)\pi(
%
  ) v \dif a\dif b\dif y\dif z.
\end{align*}
This completes the calculation of the second term \eqref{termtwoeq}.
\end{proof}
\begin{proof}[\hypertarget{termfourlemmaproof}{Proof of} \ref{termfourlemma}]
For $a,b\in \OF^\times$ , define the diagonal matrix
$$
C(a,b)=%
.
$$
Then we calculate as follows:
\begin{align*}
&q^2\int\limits_{\OF}\!\int\limits_{\p}\!\int\limits_{\OF^\times} \int\limits_{\OF^\times}
\chi(ab)\pi(t_4
%
 )
 v\dif a\dif b\dif y\dif z.\\
\end{align*}
We first focus on the integration over $z$.
\begin{align*}
&\int\limits_{\OF}
\pi(
%
) v\dif z.
\end{align*}
Substituting into the full integral, we have
\begin{align*}
&q^2\int\limits_{\OF}\!\int\limits_{\p}\!\int\limits_{\OF^\times} \int\limits_{\OF^\times}\chi(ab)
\eta^4\tau^{-1}\pi(C(a,b)
%
) v\dif a\dif b\dif y.
\end{align*}
This completes the calculation.
\end{proof}
\begin{proof}[\hypertarget{termthreepreplemmaproof}{Proof of} \ref{termthreepreplemma}]
In the following computation, we again use the invariance of $v$ under $\GSp(4,\OF)$ together with the useful identity \eqref{flipup}.
\begin{align*}
&q^2\!\int\limits_{\OF}\int\limits_{\OF}\int\limits_{\OF^\times}\int\limits_{\OF^\times} \chi(ab)
\pi(t_4
%
 )
v\dif a\dif b\dif x\dif z.
\end{align*}
This completes the preliminary calculation of the third term.
\end{proof}
\begin{proof}[\hypertarget{11lemmaproof}{Proof of} \ref{11lemma}]
We calculate \eqref{11eq} as follows:
\begin{align*}
&q\int\limits_{\OF^\times}\int\limits_{\OF}\int\limits_{\OF^\times}\int\limits_{\OF^\times} \chi(b)
\eta\pi(%
)v\dif a\dif b\dif x.
\end{align*}
This completes the calculation.
\end{proof}
\begin{proof}[\hypertarget{12lemmaproof}{Proof of} \ref{12lemma}]
We calculate \eqref{12eq} as follows:
\begin{align*}
&\int\limits_{\OF^\times}\int\limits_{\OF}\int\limits_{\OF^\times}\int\limits_{\OF^\times} \chi(b)
\eta\pi(%
 )
v\dif a\dif b\dif x.
\end{align*}
This completes the calculation.
\end{proof}
\begin{proof}[\hypertarget{13lemmaproof}{Proof of} \ref{13lemma}]
We calculate \eqref{13eq} as follows:
\begin{align*}
&q^{-1}\int\limits_{\OF}\int\limits_{\OF}\int\limits_{\OF^\times}\int\limits_{\OF^\times} \chi(b)
\eta\pi(%
)
v\dif a\dif b\dif x.
\end{align*}
This completes the calculation.
\end{proof}
\begin{proof}[\hypertarget{100eqlemmaproof}{Proof of} \ref{100eqlemma}]
We first consider only the integration over the $x$ variable and break the integral into several pieces.
\begin{align}
&\nonumber q^2\int\limits_{\OF}\pi(%
 )v\dif x\label{term1003}.
\end{align}
We now consider the integral \eqref{term1001}:
\begin{align*}
&q^2\int\limits_{\OF^\times}\pi(%
\end{align*}
Next, we consider the integral \eqref{term1002}.  For $z\in \OF^\times-(1+\p)$, it will be helpful to consider the following set $A(z)=-z^{-1}(1-z)+\p$ in order to apply the useful identity \eqref{flipup}. We calculate \eqref{term1002} as follows:
\begin{align*}
&q\int\limits_{\OF^\times}
\pi(
%
) v\dif x.
\end{align*}
Next, \eqref{term1003} is given by
\begin{align*}
&\int\limits_{\OF}\pi(%
)v\dif x.
\end{align*}
Returning to the integral \eqref{100eq}, we have
\begin{align*}
&q^2\chi(-1)\!\int\limits_{\OF^\times-(1+\p)}\int\limits_{\OF}\int\limits_{\OF^\times}\int\limits_{\OF^\times} \chi(b)
\eta\pi(D(a,b)%
)v\dif a\dif b\dif z.
\end{align*}
This completes the calculation.
\end{proof}
\begin{proof}[\hypertarget{101preplemmaproof}{Proof of} \ref{101preplemma}]
To prove the first assertion, we note that we have the matrix identity
\begin{align*}
&%
\end{align*}
where the last matrix is in $\GSp(4,\OF)$.

For the second assertion, we have a similar identity
\begin{align*}
& %
\end{align*}
where the last matrix is again in $\GSp(4,\OF)$.

For the third assertion, recall that $x \in (1-z^{-1})\varpi^{-1}+\OF^\times$ and $w=(1-x\varpi-z^{-1})\varpi^{-1} \in \OF^\times$.  Then we have the identity 
\begin{align*}
& %
\end{align*}
where the last matrix is again in $\GSp(4,\OF)$.

For the fourth assertion, recall  that $x \in (1-z^{-1})\varpi^{-1}+\varpi \OF^\times$ and $w=(1-x\varpi -z^{-1}) \varpi^{-2} \in \OF^\times$. Then we have the identity
\begin{align*}
& %
\end{align*}
where the  last matrix is again in $\GSp(4,\OF)$. 

For the final assertion, recall that $x \in (1-z^{-1}) \varpi^{-1} +\p^2$.  Then we have the identity
\begin{align*}
& %
\end{align*}
where again the last matrix is in $\GSp(4,\OF)$.
This completes the proof of the lemma. 
\end{proof}
\begin{proof}[\hypertarget{101eqlemmaproof}{Proof of} \ref{101eqlemma}]
Using the identities from Lemma \ref{101preplemma}, we calculate
\begin{align*} 
&q^2\chi(-1)\!\int\limits_{1+\p}\int\limits_{\OF}\int\limits_{\OF^\times}\int\limits_{\OF^\times} \chi(b)
\eta\pi(D(a,b)%
 )
v\dif a\dif b.
\end{align*}
This completes the calculation.
\end{proof}

\end{document}